\documentclass[12pt]{article}

\usepackage{geometry}
\usepackage{amssymb,amsmath,latexsym,amsthm,amsfonts}
\usepackage{dsfont}
\usepackage{verbatim} 
\usepackage[english]{babel}
\usepackage{cite}
\usepackage{graphicx}
\usepackage{vmargin} 
\usepackage{ifpdf} 
\usepackage{url} 
\usepackage{fancyhdr}

\usepackage{pgf}
\usepackage{tikz}
\usetikzlibrary{arrows}
\usetikzlibrary{shapes,snakes,calendar,matrix,backgrounds,folding}

\usepackage{color}
\definecolor{rltred}{rgb}{0.75,0,0}
\definecolor{rltgreen}{rgb}{0,0.5,0}
\definecolor{rltblue}{rgb}{0,0,0.75}

\usepackage{thumbpdf}
\usepackage[pdftex,colorlinks=true,urlcolor=rltred, filecolor=rltgreen, linkcolor=rltblue]{hyperref}

\theoremstyle{plain}                                                

  \newtheorem{theorem}{Theorem}[section]
  \newtheorem{lemma}[theorem]{Lemma}
  
  \newtheorem{proposition}[theorem]{Proposition}

\theoremstyle{definition}
  \newtheorem{definition}[theorem]{Definition}
\newtheorem{example}[theorem]{Example}
\newtheorem{assumption}[theorem]{Assumption}

\theoremstyle{remark}
\newtheorem{remark}[theorem]{Remark}

\newcommand{\bitem}{\begin{itemize}}
\newcommand{\eitem}{\end{itemize}}
\newcommand{\benu}{\begin{enumerate}}
\newcommand{\eenu}{\end{enumerate}}
\newcommand{\beq}{\begin{eqnarray*}}
\newcommand{\eeq}{\end{eqnarray*}}

\newcommand{\RR}{\mathbb{R}}

\newcommand{\R}{\mathbb{P}}
\newcommand{\Q}{\mathbb{Q}}
\newcommand{\qu}{\mathcal{Q}}
\newcommand{\Y}{\mathcal{Y}}
\newcommand{\E}{\mathbb{E}}
\newcommand{\EE}{\mathbb{E}}
\newcommand{\F}{\mathcal{F}}

\newcommand{\efe}{\mathbb{F}}

\newcommand{\Prob}{\mathbb{P}}

\newcommand{\ind}[1]{\mathds{1}_{#1}}

\numberwithin{equation}{section}

\title{ Robust utility maximization without model compactness 
\author{
Julio D. Backhoff Veraguas   \thanks{   Institut f\"ur Mathematik of Universit\"at Wien. Initial support by the Berlin Mathematical School, and currently by the project FWF Y782-N25, are gratefully acknowledged.
              \texttt{julio.backhoff@univie.ac.at}   } \and Joaqu\'in Fontbona 
 \thanks{ Department of Mathematical Engineering and Center for Mathematical Modeling,  UMI(2807) UCHILE-CNRS, Universidad de Chile.
              Casilla 170-3, Correo 3, Santiago-Chile. Partially supported by Fondecyt Grant 1110923,  Basal-CONICYT CMM, and  Millenium Nucleus  NC120062.
              \texttt{fontbona@dim.uchile.cl}}
 }
 }

\begin{document}
\maketitle

\abstract
We formulate conditions for the solvability of the problem of  robust utility maximization from final wealth  in continuous time financial markets,  without assuming weak compactness of the  densities of the uncertainty set, as customary in the literature. Relevant examples of such a situation typically arise   when the uncertainty set is determined through moment constraints. Our approach is based on identifying functional  spaces naturally associated  with the elements of each problem. For general markets these are modular spaces, through which we can prove a minimax equality and the existence of optimal strategies by exploiting the compactness, which we establish,  of the  image by the utility function of the set of attainable wealths. In complete markets we obtain additionally the existence of a worst-case measure, and combining our ideas  with abstract entropy minimization techniques, we  moreover  provide in that case a  novel methodology for the characterization of such measures. 

\medskip

\normalsize

\begin{description}
\item [{\bf MSC 2010}: 91G10,49N15,46E30].
\item [{\bf JEL: D81}].
\end{description}


\section{Introduction}

The problem of utility maximization in continuous time models of financial markets has been thoroughly researched in the last decades.
However, in a standard utility maximization problem one is forced to choose (or say fix) a probability measure under which the random objects in the model shall evolve. In practical terms it is next to impossible to, with complete accuracy, compute the real-world measure. For instance any statistical method shall only sign out a region of confidence for it. Therefore one is quickly led to consider utility maximization under families of possible measures (we refer to this as the \textit{uncertainty set} or \textit{set of priors}, usually denoted $\qu$) rather than over a unique a priori one; see \cite{Gil} for more on this idea. A commonly adopted (though very conservative) point of view is to look for  strategies that are optimal in the worst possible sense:
$$\mbox{maximize } \inf_{\Q\in\qu}\EE^{\Q}\left[\mbox{utility}(X)\right]\mbox{ over all admissible terminal wealths }X\mbox{ starting at }x .$$
 We will also consider  here such a point of view and, as usual in the literature, we shall refer to this stochastic optimization problem as the robust variant of the (standard, non-robust) utility maximization one.

In \cite{Quenez,Gundel,SchiedWu,Schied,FollGun}, to name a few, the problem of robust utility maximization from terminal wealth is solved in a way that greatly recovers the results known for the non-robust situation. The authors successfully apply convex-duality arguments and deliver attainability of the problem (as well as of  its dual, conjugate problem) and even the existence of what may be called a ``worst-case measure''; this is, a measure in the given family for which the optimal utility is as low as it gets. In  presence of consumption, the problem has also been considered in e.g. \cite{Ruschconsumption,Wit}. 
Robust portfolio optimization problems have also   been studied by using other tools, see e.g.  \cite{HH} for a stochastic control approach (via PDEs), as well as \cite{BMS} and the references therein for an approach using BSDEs. The case when the uncertainty set is not dominated by a single reference measure, motivated by the issue of misspecification of volatilities, was popularized by \cite{Denismax}, where it was studied under a tightness hypothesis. 

Whatever the approach,  some type of compactness  assumption on the family of possible measures seems prevalent in most of the aforementioned works, the usual assumption  in the dominated case being that the densities of the laws in the uncertainty set  form a uniformly integrable set. However, even extremely simple instances of the problem suggests that this assumption is too stringent (see Example \ref{QnotcloL0}).
 Moreover,  very little  concrete information is known about the  worst-case measure,  beyond very specific instances of the problem,  despite the fact  the dual of the robust utility maximization problem that it solves actually  is a ``convex problem'' (namely to minimize a convex functional under linear-convex constraints) when seen as an infinite-dimensional optimization problem.

In the present work, we  will restrict ourselves to the dominated case and we will only consider utilities on the positive half-line. In this setting, we will introduce a unified functional framework for the robust portfolio optimization problem that naturally  copes  with part of the   aforementioned  non-satisfactory aspects of the available literature. Our approach will be based on finding an appropriate Banach space where hypothetical worst-case measures should  a fortiori lie. 
 This space will turn out to be a convex modular space (see \cite{Musielak}), and it will be closely related to the optimization problems at hand, more concretely, to the convex dual problem related to the Legendre transform of the utility function.  In this setting, the robust utility maximization problem will reduce to solving:
$$\mbox{ maximize } \inf_{Q\in\qu}\EE^{\mathbb{P}}\left[\frac{d \Q}{d \mathbb{P}} K\right] \mbox{ over }K\in \mathcal{K}$$
where $\mathcal{K}$ is the image through the utility function of all possible terminal wealths (with common initial starting point). The crucial argument, as well as the point where most mathematical difficulties arise, is to provide  verifiable conditions on the utility function and the market under which $\mathcal{K}$ is a weakly compact set  in   the norm-dual of the mentioned modular space.
We will  rely on this approach in Theorem \ref{minimaxsinreflexividad} to prove the usual minimax equality as well as the existence of optimal wealth processes and conjugacy of value functions. We thus extend some of the results in \cite{SchiedWu,FollGun} roughly assuming that the densities of the uncertainty set be contained in the modular space and that they form a weakly closed set with respect to this topology, instead of the usual compactness assumption  (we thank a referee for pointing out that the argument in \cite{SchiedWu} for the existence of optimal wealths actually holds without compactness as well), and we do so without relying on the existence of a saddle point (the worst case measure) or on any assumption implying this. We envision that this functional point of view and the described compactness of $\mathcal{K}$ should thus open the way to new applications. Indeed, already  the characterization of worst-case measures in the complete case,  which  we will carry out in the present work, is only possible thanks to the functional setting we adopt, and the compactness of $\mathcal{K}$ has been crucially applied in \cite{BSsensibilidad2} in the context of sensitivity analysis.

 When aiming to recover  those results in \cite{SchiedWu,FollGun} not covered by our Theorem \ref{minimaxsinreflexividad}, for instance the existence of a worst-case measure, we realize that replacing the usual compactness assumption by reflexivity of the modular space is a sufficient condition to do this. In this respect we prove, modulo some pathologies on the filtered probability space, that our modular spaces are unfortunately never  reflexive for strict incomplete markets; this is the content of Theorem \ref{tristeintro} (more specifically Theorem \ref{triste} and the remarks thereafter).
 
  On the positive side, when we specialize our analysis to complete markets, our modular spaces become Orlicz-Musielak spaces  and we can provide easily verifiable conditions under which they become reflexive. Of course, Orlicz spaces are well known about in Mathematical Finance (see e.g.\  \cite{Cheridito} regarding risk measures, \cite{Gol} on utility maximization and \cite{Biagini} on admissibility of trading strategies). Related to our work, in \cite{FollGun,Gundel} Orlicz spaces arise  in connection to the  Vallee Poussin criterion  when studying the problem by means of f-divergences, and we will comment more about this in Section \ref{comparison}. Our choice of an Orlicz space, in the complete case, obeys different   considerations  and  makes a more systematic use of the properties of the space in connection to the robust problem; furthermore, our functional setting  will be crucial for the new application which we have already hinted at and which we discuss next.
    
 Using our Orlicz space  formulation of the dual (minimization) problem for complete markets we will give, in the reflexive setting, a novel and  explicit characterization of the worst-case measure that covers  a much broader range of applications than is available in the literature. 
More precisely, by writing the general set of possible models $\qu$ in terms  of a potentially infinite system of linear constraints (that may  be thought of  as moment constraints on some market observables or insider information), we will be able to adapt to the financial framework  some general entropy minimization  techniques developed in \cite{LeoAbstract,Leo-ent} and   characterize  in  Theorem \ref{teocompletocaractintro}   the worst-case measure $\hat{\Q}\in\qu$ in terms of a related abstract concave  maximization problem. We may call it with some abuse the \textit{dual of a dual} problem. By finding  a  solution $g$   to  that problem we obtain the expression: 
$$\hat{\Q}=  \mbox{risk-neutral density} \times [U^{-1}]'  \left(\mbox{linear operator}(g)\right),$$
where the \textit{linear operator} above describes how the element $g$ acts upon the observables of the market  that we use to describe  (through their moments)   the set $\qu$.
The so-called  dual of a dual problem may in many practical situations be easier to solve than the original one; for instance,   it is finite-dimensional if $\qu$ is specified by finitely many constraints.

The remainder of the paper is organized as follows. We start  Section \ref{dualidad} describing the mathematical framework of the  robust optimization problem in continuous-time financial markets following \cite{SchiedWu} and introduce the basic notation required throughout. Then in Section \ref {motivgen} we will  state  our main results about incomplete markets, in \ref{comparison} we compare them with the existing literature and finally in \ref{maincomplete} state our specialized results for complete markets. In Section \ref{OMtopo} we recall some known properties of Orlicz-Musielak spaces and in \ref{secapli} provide results of our own connecting them to our robust problem. Our main results on the robust optimization problem in the complete case are then established and proven in Section \ref{worstcasecomplete}.  In Section \ref{incompletecase}
we introduce the modular spaces associated with the incomplete case and study them extensively. Both Sections \ref{worstcasecomplete} and \ref{incompletecase} are independent of each other, and the reader can skip either of them depending on which result he/she is interested in. Finally, some technical facts are proved in the Appendix.

\section{Preliminaries and statement of main results }
\label{dualidad}

We will work {in a similar setting} as \cite{SchiedWu,KrSch}. Let there be $d$ stocks and a bond, normalized to one for simplicity. Let $S=\left(S^i\right)_{1 \leq i \leq d}$ be the price process of these stocks, and $T<\infty$ a finite investment horizon. The process $S$ is assumed to be a  semimartingale in a filtered probability space $(\Omega,\efe,(\F_t)_{t\leq T},\R)$, where  $\R$ will always stand for the \textit{reference measure}. The expectation with respect to $\R$ will be denoted by  $\E$. The set of all probability measures on $(\Omega,\efe)$ absolutely continuous w.r.t $\R$ will be denoted by ${\cal P}$, and  the expectation with respect to $\Q\in {\cal P} \backslash\{ \R\}$ will be expressed by $\E^{\Q}$.
\medskip

A (self-financing) portfolio $\pi$ is defined as a couple $(X_0,H)$, where $X_0\geq 0$ denotes the (constant) initial value associated to it and $H=(H^i)_{i=1}^d$ is a  predictable and $S$-integrable process  which represents  the number of shares of each type under possession. The wealth associated to a portfolio $\pi$ is  the process $X=(X_t)_{t\leq T}$ given by
\begin{equation}
X_t=X_0+\int_0^t H_u dS_u \label{riqueza}
\end{equation}
and the  set of attainable wealths from $x$ is defined as
\begin{equation}
\mathcal{X}(x)=\left\{ X\geq 0: X \mbox{ as in } \eqref{riqueza}\mbox{ s.t. } X_0 \leq x \right\}. \label{admisibles}
\end{equation}
The set of equivalent local martingale measures (or risk neutral measures) associated to $S$ is 
\begin{equation}
\mathcal{M}^e(S)=\left\{\Prob^* \sim \R:\mbox{ every } X \in \mathcal{X}(1) \mbox{ is a }\Prob^*\mbox{-local martingale}  \right\} \label{me}
\end{equation}
which reduces to \begin{equation*}
\mathcal{M}^e(S)=\left\{\Prob^* \sim \R:S \mbox{ is a }\Prob^*\mbox{-local martingale}  \right\} 
\end{equation*}
if $S$ is locally bounded. This is assumed in all the sequel,  together with the fact that the market is \textit{arbitrage-free} in the sense of NFLVR, meaning that $\mathcal{M}^e(S)$ is not empty. \\
As usual the market model is coined \textit{complete} if $\mathcal{M}^e(S)$  is reduced to a singleton, i.e. $\mathcal{M}^e(S) =\{\Prob^*\}$. 
Given  $\Q\in {\cal P}$, the following set   generalizes the set of density processes (with respect to  $\Q$)  of  risk neutral measures equivalent to it:
\begin{equation*}
\mathcal{Y}_{\Q}(y):= \left\{Y \geq 0| Y_0=y \mbox{ , } XY \mbox{ is } \Q-\mbox{supermartingale } \forall X \in \mathcal{X}(1)  \right\}. \label{YQ}
\end{equation*} 
 Introduced in \cite{KrSch},     $\mathcal{Y}_{\Q}(y)$ plays a central role in portfolio optimization in incomplete markets. 
   
   \medskip 
   
\begin{definition} \label{INADA}
A function $U:(0,\infty)\rightarrow \RR$ is called a \textit{utility function on} $(0,+\infty)$, if it is strictly increasing, strictly concave and continuously differentiable. It will be said to satisfy \textit{INADA} if $$U'(0+)=\infty \mbox{ and }U'(+\infty)=0\, . $$ Such a function $U$ is always extended  as $-\infty$ on $(-\infty,0)$.  Its  \textit{asymptotic elasticity}, introduced in \cite{KrSch}, is defined as $AE(U):= \limsup_{x \rightarrow \infty} \frac{xU'(x)}{U(x)}.$ 
Last, if   $\Delta:=\lim_{x\to+\infty}U(x)<\infty$,  we set $U^{-1}=+\infty$ on $[\Delta,\infty)$.

\end{definition}

Suppose  now that an agent aims to optimize the utility $U$ of her final wealth, by investing during a time interval $[0,T]$ in a market  which might be described by more than one probabilistic model (the actual or more accurate one being unknown to her).  Let $\qu\subset {\cal P}$ be a set of  feasible  probability measures on  $(\Omega,\efe,(\F_t)_{t\leq T},\R)$ representing the mentioned ambiguity or uncertainty. We shall refer to such a set as the \textit{uncertainty set} from here on. A common paradigm is that  the agent tries to maximize the worst-case expected utility given the set of models under consideration,  by solving 
the optimization problem 
\begin{equation}\label{RobPortOpt}
\sup_{X \in \mathcal{X}(x)} \inf_{\Q \in \qu} \E^{\Q}\left(U\left(X_T\right)\right),
\end{equation}
(a suitable  meaning can  be given to  the expectation in case $U$ is unbounded). Throughout the present work it will be  assumed that $\qu$ contains only probability measures that are absolutely continuous with respect to $\R$.  We will write 
$$\qu_e := \{\Q \in \qu|\Q \sim \R\}$$
and respectively denote by 
$$\frac{d\qu}{d\R}:=
\left\{ \frac{d\Q}{d\R}: \Q \in \qu \right\}\, , \quad
\frac{d\qu_e}{d\R}:=
\left\{ \frac{d\Q}{d\R}: \Q \in \qu_e \right\}=\left\{ \frac{d\Q}{d\R}\in  \frac{d\qu}{d\R} : \frac{d\Q}{d\R}>0 \mbox{ a.s. } \right\}.$$
  the set of densities  with respect to $\R$  of the elements of $\qu$ and $\qu_e$.
As in the standard, non-robust, setting (see \cite{Pham} for general background), the   dual formulation of  the  optimization problem \eqref{RobPortOpt}  will make use of the  conjugate function of $U$,  given by
\begin{equation*}
V(y):=\sup_{x>0}[U(x)-xy] \mbox{  }\forall y>0 \label{V} \\
\end{equation*} 
(actually the Fenchel conjugate of $-U(-\cdot)$). The following functions  commonly used in the literature to tackle problem   \eqref{RobPortOpt}, will also be relevant  here:
 \begin{align}
u(x)&=\sup_{X \in \mathcal{X}(x)}\inf_{\Q \in \qu} \E^{\Q}\left(U\left(X_T\right)\right) &
u_{\Q}(x)&=\sup_{X \in \mathcal{X}(x)} \E^{\Q}\left(U\left(X_T\right)\right), \notag\\
v_{\Q}(y)&= \inf_{ Y \in \mathcal{Y}_{\Q}(y)}\E^{\Q}\left(V\left(Y_T \right) \right) &
v(y)&= \inf_{\Q \in \qu_e} v_{\Q}(y)  	. \label{v}
\end{align}
Of course,  $u_{\Q}(x)$ is the investor's \textit{subjective} utility   under model $\Q\in\qu_e$, when starting from an initial wealth not larger that $x> 0$, whereas $u(x)$  is her robust utility.
The function $x\mapsto u_{\Q}(x)$ is concave (as an easy check shows),  so that $u_{\Q}(x_0)<+\infty$ at some $x_0>0$  for some given $\Q\in \qu$ implies $u_{\Q}<+\infty$ and  then,  $u<+\infty$, by the usual min-max inequality.

For a fixed $\Q\in \qu_e$ it was proven in Theorem 3.1 of \cite{KrSch}  that $u_{\Q}$ and $v_{\Q}$ are conjugate:
  \begin{equation}\label{uQvQ}
u_{\Q}(x)= \inf_{y>0} \left(v_{\Q}(y)+xy \right)\,\, \,\,\,\,\,\,\mbox{   and    } \,\,\,\,\,\,\,\,v_{\Q}(y)= \sup_{x>0} \left(u_{\Q}(x)-xy \right)
 \end{equation}
whenever     $u_{\Q}$ is finite. Hence, since the inequalities 
 \begin{equation}\label{candconj}
 \begin{split}
u(x)\leq &  \inf_{y>0} \left(   \inf_{\Q \in \qu} \inf_{ Y \in \mathcal{Y}_{\Q}(y)}\E^{\Q}\left(V\left(Y_T \right) \right)    +xy \right)\\
\leq & \inf_{y>0} \left(   \inf_{\Q \in \qu_e} \inf_{ Y \in \mathcal{Y}_{\Q}(y)}\E^{\Q}\left(V\left(Y_T \right) \right)    +xy \right) = \inf_{y>0} \left( v(y)  +xy \right) \\
\end{split}
  \end{equation}
 always hold, the function $v$  can be considered as  a candidate conjugate of $u$. 

We will  denote in the sequel  by $L^0=L^0(\Omega,\R)$  the space of measurable functions equipped with the  topology of convergence in probability, and by $L^0_+\subset L^0$ the cone of non-negative functions therein. We shall also write $$\Y:=\Y_{\R}(1),$$ and we will often use $Y$ instead of $Y_T$, which should be clear from context.  In order to state the  assumptions that will hold throughout this work we will also need the  subset
\begin{equation}\label{Yutil}
\Y^*:=\{Y\in\Y:Y>0 \mbox{ a.s. and }\forall \beta>0, \E[V(\beta Y)]<\infty\}.
\end{equation}

\subsection{Main results in general  markets }
\label{motivgen}

We start noting that for every $\Q \in \qu_e$, we have $\mathcal{Y}_{\Q}(y) = \left\{\frac{yY}{Z^{\Q}}: Y \in \mathcal{Y}_{\R}(1) \right\}$, where $Z^{\Q}$ is the density process of $\Q$ w.r.t.\ $\R$, hence%
\begin{equation}\label{exprv1}
v(y)=  \inf_{\Q \in \qu_e} \inf_{Y \in \mathcal{Y}_{\R}(1)} \E^{\R}\left[ \frac{d\Q}{d\R} V\left(y Y_T \left[\frac{d\Q}{d\R}\right]^{-1}\right) \right] .
\end{equation}
Thus, if $v$ is to be finite at some point $y>0$, the only measures $\Q$ that matter in \eqref{exprv1}  are those such that, for some  $Y \in \mathcal{Y}_{\R}(1)$,    $$ \E^{\R}\left[ \frac{d\Q}{d\R} V\left(y Y_T \left[\frac{d\Q}{d\R}\right]^{-1}\right) \right] <\infty.$$

This motivates us to restrict from the outset the set $\qu$ to consist of measures $\Q$  for which $ \frac{d\Q}{d\R}$ is in the space   of measurable functions
 \begin{equation*}
	L_I\,\,: =\,\,  \left\{ Z\in L^0 \mbox{ s.t. } \exists \alpha >0,  \inf_{Y\in \Y} \E[\vert Z \vert V(Y/ (\alpha \vert Z\vert)] < \infty \right\}\, \,= \, \,  \bigcup\limits_{Y\in \Y} L_{   | \cdot| V 	\circ Y/ | \cdot | }\, ,
 \end{equation*}
where for every $Y\in \Y$ we define:
    $$L_{   | \cdot| V 	\circ Y/ | \cdot |  }:= \left\{Z\in L^0 \mbox{ s.t. } \exists \alpha >0,
\E^{\R}\left[ | Z|V(Y/( \alpha | Z| ) )\right]  < \infty \right\}.
 $$ 

We will see in Section  \ref{Orlicz} that the function
 $z\mapsto |z| V ( Y/ | z |) $  is a.s.\  non-negative and convex and that   $L_{   | \cdot |   V 	\circ Y/ | \cdot |  }$  turns out to be an Orlicz-Musielak space (see Remark \ref{dualKoz}) for each  $Y\in \Y^*$.   In particular, it is  a  Banach space with  the  adequate norms; properties of  these spaces (which can be seen as Orlicz spaces based on ``random Young functionals'')  will be recalled in Theorem \ref{variasobsorlicz}. The convex conjugate of  $   | \cdot |  V 	\circ Y/ | \cdot | $ will be shown in Lemma \ref{lemagama} to be the function   $ Y U^{-1}  \circ| \cdot|$, and  it will play a pre-eminent role, as will do the associated Orlicz-Musielak space 
   $$L_{YU^{-1}  \circ| \cdot| }\,\,:= \,\,\left\{Z\in L^0 \mbox{ s.t. } \exists \alpha >0,
\E\left[Y U^{-1} \left( \alpha | Z | \right) \right]  < \infty \right\}.$$

The following assumption will be relevant in the study of  topological duality between  the spaces   $L_{   | \cdot |   V 	\circ Y/ | \cdot |  }$ and $L_{YU^{-1}  \circ| \cdot| }$. It is not assumed to hold, unless specifically stated: 
 \medskip
\begin{assumption}
$ $
\label{Usuposdelta} For some constants $a,b,k,d> 0$, the convex  functions $V(\cdot) $ and $U^{-1}(\cdot)$ on $(0,\infty)$ satisfy for all $y>0$:
\begin{align} \label{Vdelta}
V(y/2) &\leq a V(y)+b(y+1)  \\
U^{-1}(2y) &\leq k U^{-1}(y)+d . \label{U-1delta}
\end{align}
\end{assumption}
In  the jargon of Orlicz space theory (see e.g. \cite{RaoRen}), Assumption \ref{Usuposdelta} correspond  to ``$\Delta_2$  and $\nabla_2$''  conditions on the  Young function $ | \cdot |   V 	\circ 1/ | \cdot |  $.  As pointed out in Theorem \ref{OMreflex}, Assumption \ref{Usuposdelta} implies  reflexivity of $L_{   | \cdot |   V 	\circ Y/ | \cdot |  }$ and $L_{YU^{-1}  \circ| \cdot| }$, and is necessary for the latter if $\Prob$ is atomless.

 The space   $L_I$  will be endowed with a suitable Banach space topology  called Modular Space topology which  generalizes  the Orlicz-Musielak one  (see Section \ref{mod intro})  and   tightly harmonizes with our optimization problem. The search for verifiable conditions on the function $U$ that may render the space $L_I$ to be tractable will lead us to introduce the   space 
\begin{equation*}
 	L_J: =  \left\{ Z\in L^0 \mbox{ s.t. } \exists \alpha >0,\;  \sup_{Y\in \Y} \E\left[Y U^{-1} \left( \alpha | Z | \right) \right] < \infty \right\}\subseteq   \bigcap\limits_{Y\in \Y} L_{YU^{-1}  \circ| \cdot| } .
	 \end{equation*}
	 	 In general $L_J$ is included in the algebraic dual of $L_I$. Under Assumption \ref{assgnral} below $L_J$ is actually included in the topological dual of $L_I$ (cf.\ Proposition \ref{Hoelder}), and if additionally \eqref{Vdelta} holds, the latter space and  $L_J$ will be precisely isometric isomorphic (cf.\ Proposition \ref{rep}). Denoting  by  $\sigma(L_I,L_J)$ the  weak topology on $L_I$ induced by $L_J$, we now  state  our main hypothesis:

 \medskip
 
\begin{assumption}\label{assgnral}
\begin{enumerate}
\item $U$ is a utility function    on $(0,\infty)$ satisfying INADA and  such that $U(0+)=0$. 
\item The set $\Y^*$ is a non-empty subset of $\Y_{\R}(1)$. 
\item Regarding $\qu$ we assume:
\begin{itemize}
\item[(a)] $\qu$ is countably convex.
\item[(b)] $[\R(A)=0 \iff \forall \Q \in \qu,\Q(A)=0]$.
\item[(c)] $d\qu/d\R$ is a non-empty $\sigma(L_I,L_J)$-closed subset of $L_I$
\item [(d)]$\exists x>0,\Q\in\qu_e$ such that $u_{\Q}(x)<\infty$
\end{itemize}
\end{enumerate}
\end{assumption}
\medskip

Exploiting a certain compactness of the image under $U$ of the terminal wealths, as elements in $L_J$,  our main result for general markets, proved  in Section \ref{subsecmain},  will establish  the   minimax equality and the existence of optimal strategies: 
\medskip

\begin{theorem}
\leavevmode 
\label{minimaxsinreflexividad}
Suppose  Assumption \ref{assgnral} holds. Assume moreover that $L_I^*\cong L_J$, which is true as soon as \eqref{Vdelta} in Assumption \ref{Usuposdelta} additionally holds and in particular if $AE(U)<1$. Then for every $x>0$: 
\begin{eqnarray}
u(x)&=\inf_{\Q \in \qu} \sup_{X \in \mathcal{X}(x)} \E^{\Q}\left(U\left(X_T\right)\right) &= \inf_{\Q \in \qu}  \E^{\Q}\left(U\left(\hat{X}_T\right)\right) 
\notag\\
 &=\inf_{\Q \in \qu_e} \sup_{X \in \mathcal{X}(x)} \E^{\Q}\left(U\left(X_T\right)\right)&<+\infty, \label{igualdades}
\end{eqnarray} 
for some $\hat{X}\in \mathcal{X}(x)$. Moreover $v$ is finite and $u,v$ are conjugate on $(0,\infty)$.\\
Furthermore if $L_I$ is reflexive, which happens as soon as the market is complete and the full Assumption \ref{Usuposdelta} holds, then there is a saddle point, i.e.\ there exists  a unique $\hat{\Q}\in\qu$ so that all the  values in \eqref{igualdades} equal to $\E^{\hat{\Q}}\left[U\left(\hat{X}_T\right)\right]$.
\end{theorem}
\medskip

Our second main result  in the setting of incomplete markets, however, is of a negative kind. It states that reflexivity of $L_I$ is virtually impossible in most strict incomplete market models, independently of how good the utility functions is. This is quite remarkable since it implies that the route, through reflexivity, to establish the existence of a saddle point when the set $\qu$ is only weakly closed in $L_I$ (as by the end of the previous theorem), is feasible if and only if the market is complete to begin with:

\medskip

\begin{theorem}\label{tristeintro} 
Under parts 1.\ and 2.\ of Assumption \ref{assgnral}, if the set $\Y$ is not uniformly integrable, then $L_I$ cannot be reflexive.
\end{theorem} 

\medskip 

As it shall be discussed in Section \ref{incompletecase},  in most reasonable strictly incomplete market models (for instance those involving the brownian filtration)  $\Y$ is indeed never uniformly integrable.

 In the complete case, in turn, $\Y$ has of course a maximal integrable element for the a.s.\ order   (see e.g.\ Lemma 4.3 in \cite{KrSch}) and therefore the previous result does not preclude reflexivity in that case. Indeed, in the complete case and under Assumption \ref{Usuposdelta}, one obtains from the proof of Theorem \ref{minimaxsinreflexividad} that $L_I$ is a reflexive Orlicz-Musielak  space and so, owing to the existence of a saddle point, one can be more specific about the solution to the robust optimization problem. This is done in Theorems 2.2.5 and Section 2.4.1 of the thesis \cite{JBtesis}, whereby the author relates the dual and primal optimizers as expected from e.g.\ \cite{SchiedWu}.  However, this is  not the point about complete markets we want to stress in this article; our main contribution in such case is that,  thanks to our functional analytical approach and the nice Orlicz  space structure it leads to in the complete setting, we are able to provide in a  systematic way 
  the characterization of  the saddle-point element $\hat{\Q}$,  that is, the \textit{worst} element in $\qu$ for each utility function.  
  
  Before detailing our specific results for complete markets in Section \ref{maincomplete}, let us discuss   our assumptions and the relationship between Theorem \ref{minimaxsinreflexividad} and  results in the existing literature.

\subsection{Discussion on Assumption \ref{assgnral} and comparison to the existing literature} 
\label{comparison}
$ $
\medskip

\noindent\textbf{On the utility function}: Condition $U(0+)=0$ was assumed as it implies the desirable property $V\geq 0$. If $U$ were bounded from below our results would still hold. The following are examples of utility functions for which our results apply:

\begin{example}
Power utilities \ $U(\cdot)=\alpha^{-1}(\cdot)^{\alpha}$, $\alpha\in (0,1)$ fulfill
 point 1.\ in Assumption \ref{assgnral}. Moreover, this assumption is satisfied if and only if    $U^{-1}$ is convex and increasing, $U^{-1}(0+)=0$, $[U^{-1}]'(0+)=0$ and $[U^{-1}]'\left(\lim_{x\to\infty}U(x)\right)=\infty$. So for instance the inverse on $[0,+\infty)$ of
$x\mapsto e^x-x-1$ satisfies it as well. Power utilities, as described, also satisfy Assumption 
\ref{Usuposdelta}.
\end{example}\medskip

\noindent\textbf{On the non-emptiness of $\Y^*$}  The reason behind point 2.\ of Assumption \ref{assgnral}  is two-fold. It ensures that the result in Remark \ref{dualKoz} stating that the Orlicz-Musielak spaces $L_{   | \cdot| V 	\circ Y/ | \cdot |  }$ and $ L_{YU^{-1}  \circ| \cdot| }$ are well-behaved whenever $Y\in\Y^*$, be lifted to the spaces $L_I$ and $L_J$. On the other hand, it precludes the combinations of market models and utility functions for which, even in the non-robust case, primal optimizers do not exist; we come back to this under the point ``Global comparison of our assumptions''.

\medskip

\noindent\textbf{On the topological constraint on $\qu$}:  Our point 3.(c) in Assumption \ref{assgnral}, specifically $d\qu/d\R\subset L_I$, implies that $\forall \Q\in\qu,\exists y>0,v_{\Q}(y)<\infty$. Under the assumption $L_I^*\cong L_J$  in our Theorem \ref{minimaxsinreflexividad}, we further have that $\forall \Q\in\qu,\forall y>0,v_{\Q}(y)<\infty$. This is a strenthening of Condition $(2.10)$ in \cite{SchiedWu}, which the authors there use to prove existence of optimal wealth processes. Furthermore, in Theorem 2.2.\ of \cite{SchiedWu} the authors succeed in proving conjugacy of the value functions without anything like our condition $d\qu/d\R\subset L_I$, but in turn suppose the stronger $L^0$-closedness condition,  typically assumed in the literature (see e.g.\ \cite{SchiedWu},\cite{FollGun}) and equivalent in the present context to weak $L^1$ compactness. By Proposition \ref{Hoelder} below, our weak-closedness condition for $d\qu/d\R$ is indeed implied  by  the usual closedness in $L^0$ and,
as the following example exhibits, the converse is not true. Furthermore, the same example shows that  in our setting  a ``least-favourable'' measure might not exist, contrary to the framework of \cite{Bau},\cite{Schied}:

\begin{example}\label{QnotcloL0}
Assume an investor knows or anticipates that the mean of a ${\cal F}_T-$measurable unbounded random variable $h$ (e.g. $h=S_T$) is bounded from below by a constant $A>0$. If   $\E(h)<\infty$, then the set of densities $\frac{d \qu_A}{d \Prob}$ of the set $\qu_A:= \{\Q \in {\cal P}: \Q \ll \Prob,   \, \E^{\Q}(h)\geq A\} $  is not closed in $L^0$.  Indeed, the sequence  $\Q^n(\cdot) : = \Prob(\cdot | h\geq nA)\in \qu_A $,  is  such that $ \frac{d \Q^n}{d \Prob}=\Prob( h\geq nA) ^{-1}\ind{\{  h\geq nA\}}\to 0$ in $L^0$ when $n\to \infty$, yet obviously $0 \notin \qu_A$. Consider now the utility function  $U(x)=\frac{x^{\alpha}}{\alpha}$, $\alpha \in (0,1)$, so that after some computations we see $  L_{  | \cdot |    V 	\circ 1/ | \cdot |  } =L^{\frac{1}{\alpha}}$,  and call $\tilde{\qu}_A:= \{\Q \in {\cal P}: \Q \ll \Prob,\, d\Q/d\Prob \in  L^{\frac{1}{\alpha}}, \,\E^{\Q}(h)\geq A\} $, which by the same argument is not closed in $L^0$. If however $h$ is an element of $L^{\frac{1}{1-\alpha}}$,  one can check that   $ \tilde{\qu}_A$ is a closed subset of $  L_{  | \cdot |    V 	\circ 1/ | \cdot |  }$.  Finally, it is not difficult to see with the aid of Lagrange multipliers and under given conditions on $h$ and $A$, that the solution of $\inf\{\E[Z^2]:Z\in d\tilde{\qu}_A/d\Prob\}$ is a linear function of $h$, whereas the solution of $\inf\{\E[Z^{3/2}]:Z\in d\tilde{\qu}_A/d\Prob\}$ is a quadratic function of $h$. In particular then $\tilde{\qu}_A$ has no least-favourable measure in the sense of e.g.\ \cite{Bau},\cite{Schied},\cite{FollSchie}.\end{example}  

Finally, since we cannot get countable convexity out of convexity with our weak-closure assumption, this condition has  been put in Assumption \ref{assgnral}. Had we required $\Q_e\neq\emptyset$ instead, we could have assumed usual convexity. Condition 3.(d) therein, which we add straight from the beginning,  is required in any case for all the results in the literature. 

\medskip 

\begin{remark}
Our motivation for the set $\qu$ comes from modeling concerns, namely we want to account for anticipation of moments of observables or insider information of statistical kind, rather than from axiomatic considerations. Regarding these, one could consider $Z\in L_J\mapsto \rho_{\qu}:=-\inf_{\Q\in\qu}\EE^{\Q}[-Z]$ as a coherent risk measure and survey its properties in terms of conditions on $\qu$, and conversely ask which coherent risk measures on $L_J$ can be obtained as $\rho_{\qu}$ for $\qu$ in some appealing class. In general the dual representation of a coherent risk measure via a not necessarily compact set is only equivalent to lower semicontinuity, and for nice lattices as our Modular spaces, to the Fatou property; see \cite{BiaFri} on the so-called C-property, which holds in our setting. If $ L_J$ were an Orlicz heart, one can find in \cite[Corollary 7]{OrihuelaGalaz} (resp.\ \cite[Corlollary 4.1]{Cheridito}) the equivalence between continuity from below of $\rho_{\qu}$ (resp.\ Lipschitz continuity) and weak compactness (resp.\ norm boundedness) in the dual space of $d\qu/d\Prob$. Along similar lines, one could ask what kind of preferences can be numerically represented via ``expected utility under multiple priors'', as in \cite{Gil} or \cite{Italians}, and further investigate the properties of this representations in terms of the set of priors $\qu$. 

\end{remark}
\medskip

\noindent\textbf{Global comparison of our assumptions}: Comparing our results with those in \cite{SchiedWu} and \cite{FollGun}, which we take as benchmark, we find that we must require stronger integrability of the elements in $d\qu/d\Prob$, better-behaved utility functions, and the more stringent dual finiteness condition $\Y^*\neq\emptyset$. If the set $\qu$ were closed in $L^0$ our assumptions would be then unnecessarily demanding. The point is, we prove  min-max equality and duality of value functions beyond the $L^0$-closedness assumption. 

We stress that Theorem \ref{minimaxsinreflexividad} can be applied to the classical, non-robust situation as well, asserting the existence of an optimal wealth process and the finiteness everywhere of the value function, under the assumptions that $U$ is a utility function on $(0,\infty)$ bounded from below and satisfying INADA, and that $ \Y^*\neq \emptyset$ (see \eqref{Yutil}). These conditions are not necessary for the existence of optimal wealths, and as sufficient conditions they are stronger than the one given in \cite{KrSch2} (namely finiteness of the dual value function). The point is, our modular space proof is purely functional-analytical (see Proposition 2.5.6 in \cite{JBtesis} for a self-contained proof not relying on Theorem \ref{minimaxsinreflexividad}) and the condition on $\Y^*$ precludes market models as in Example 5.2 in \cite{KrSch}, whereby the failure of $AE(U)<1$ implies non-existence of optimal wealths. This functional-analytical proof is only seemingly shorter or neater than the classical one (relying in convex-compactness of the solid hull of $\mathcal{X}(1)$), since it relies on the fact that $L_J$ is a norm-dual space (see Proposition \ref{rep}), which is lengthy to prove. Even taking this fact from granted, one still needs to use the bipolar theorem, which is in any case necessary to prove the mentioned property of the solid hull of $\mathcal{X}(1)$.

\medskip

\noindent\textbf{Related use of Orlicz spaces in the literature}: We compare our Modular space approach (Orlicz-Musielak space, in case of complete markets) to that of \cite{FollGun,Gundel}, which to the best of our knowledge are the only works where Orlicz spaces are used for the problem of robust utility maximization. In their setting the set $\qu$ is assumed weakly compact in $L^1$ and therefore, out of Valle-Poussin criterion, it is bounded in an Orlicz space induced by a well-behaved Young function. Starting from this the authors eventually prove that the set of martingale measure densities possibly contributing to the dual problem is also weakly compact in $L^1$, by means of constructing a second Young function. This establishes full dual attainability under the hypotheses given. We in turn look for compactness elsewhere, in the set of images of the terminal wealths through the utility function, and do not ask for compactness of $\qu$. We do establish that the relevant elements for the dual problem, say $\tilde{\qu}\subset\qu$, are bounded in the modular space (see Remark \ref{ilusion}), but this is a long way from implying that they form a weakly compact set in $L^1$. It is only for complete markets under Assumption \ref{Usuposdelta} and under the condition that $\Prob$ is a martingale measure, that our spaces become faintly comparable to those of \cite{FollGun}. In such case, as the proof of Lemma 2.11 or Remark 2.13 in \cite{FollGun} reveals (taking $f=g$ therein), the classical Orlicz space $L_{   | \cdot| V 	\circ 1/ | \cdot |  }$ we get has a strictly stronger topology than any of the Orlicz spaces introduced there, and cannot be expressed in terms of any of them.
 
\subsection{Main results in the complete case: characterization  of the worst-case measure}\label{maincomplete}
$  $
\medskip

If in the complete case   we denote by $Y^*$ the density of the unique equivalent martingale measure w.r.t.\ $\Prob$,  it is easy to see that the spaces $L_I$ and $ L_{   | \cdot |   V 	\circ Y^*/ | \cdot |  }$, defined in Section \ref{motivgen}, must coincide. We may thus assume without loss of generality  that the reference measure $\Prob$, whose single role is to specify null sets, is a martingale measure and so $Y^*\equiv 1$ (this does not trivialize our robust problem as measures in $\qu$  need not to be martingale measures even if $\Prob$ is). Under this assumption, Lemma 4.3 in \cite{KrSch} and its proof states that every terminal value of the elements $Y \in \mathcal{Y}_{\R}(1)$ is bounded by $1$ and (since $V$ is non-increasing) we have:
\begin{equation}\label{vcompleto}
v(y)=\inf\limits_{Z \in \frac{d\qu_e}{d\R}} \E\left[ZV\left(\frac{y}{Z} \right)\right].
\end{equation} 
By the same argument, the functions space $L_I$ we worked with in Section \ref{motivgen} is $ L_{   | \cdot |   V 	\circ 1/ | \cdot |  }$;  a classical Orlicz space. The space   $L_J$ corresponds accordingly to the Orlicz space $L_{U^{-1}  \circ| \cdot| }$. All in all, we can write:
\begin{equation}\label{dualcomplete}
v(y)= \inf \big\{  \E (\gamma^*_y(Z)) : \, Z\in   L_{   | \cdot |   V 	\circ 1/ | \cdot |  } \mbox{ s.t. }
\E^{\R}(Z)=1 \mbox{  and } Z \cdot d\R\in \qu \big\}, 
\end{equation}
where  $\gamma^*_y(\cdot)$  is the convex function which equals $z\mapsto z V\left(\frac{y}{z}\right)$ for $z>0$  and $+\infty$ otherwise
(see Lemma  \ref{lemagama} for further  properties of this  $\gamma^*_y$).
  
  \smallskip 
    
 We will assume  that the  set  $\qu$ is defined by the constraint that the under each element $\Q\in \qu$, the average value of a given observable of the market $\theta$, with values in a (possibly infinite dimensional) vector space, lies on a prescribed convex subset of this space.  More precisely, we will consider 
\begin{itemize}
\item [i)]   $(\mathbf{F}_0,\mathbf{G}_0)$    a pair of linear spaces  of arbitrary dimension, with $\mathbf{F}_0$ the algebraic dual of $\mathbf{G}_0$ and with dual product denoted
 $\langle \cdot, \cdot\rangle_{\mathbf{G}_0,\mathbf{F}_0}$.
  \item[ii)]   $\theta:\Omega\to  \mathbf{F}_0$   a function (or ``observable'') on the market with values in $\mathbf{F}_0$ and
 \item[iii)]   $\mathbf{C}_0\subset \mathbf{F}_0$   a convex subset .
\end{itemize}
\medskip

  In this setting, we will  characterize  the worst-case measure  $\hat{\Q}$    using techniques  for  the minimization of  abstract entropy  functionals,   developed   in the series of papers  \cite{Leo-ent,LeoAbstract,Minvery}.  Following those works we will   make
 \medskip
 
\begin{assumption}\label{suposQLeo}
$ $
\begin{itemize}
\item[i)] $\forall g \in \mathbf{G}_0$, the function $\omega\in\Omega \mapsto \langle g, \theta(\omega) \rangle_{\mathbf{G}_0,\mathbf{F}_0}$ is measurable.
\item[ii)]   $\forall g \in \mathbf{G}_0 , \left\langle g,\theta \right\rangle_{\mathbf{G}_0,\mathbf{F}_0} \in L_{U^{-1}  \circ| \cdot| }$.
\item[iii)] $\forall (g,a) \in \mathbf{G}_0\times \RR$, one has $\langle g, \theta(\cdot) \rangle_{\mathbf{G}_0,\mathbf{F}_0} = a  \, \R$ -a.s.  iff $g=0 $ and $ a=0$.
\item[iv)] We have $\qu \neq \emptyset $ and 
  \begin{equation*}
  \frac{d\qu}{d\R}  =\left\{ Z  \in   L_{   | \cdot |   V 	\circ 1/ | \cdot |  }: 	\,  Z\geq 0 \mbox{ a.s.,}  \, \EE(Z)=1  \mbox{ and }\Theta\left(Z\right) \in \mathbf{C}_0  \right\},
\end{equation*} 
where  $\Theta:    L_{   | \cdot |   V 	\circ 1/ | \cdot |  } \to \mathbf{F}_0$ denotes the linear   operator $\Theta(Z)= \int \theta Z\, d\R$ such that
$$\left\langle g,\Theta(Z) \right\rangle_{\mathbf{G}_0,\mathbf{F}_0} = \int_{\Omega} \left\langle g, \theta \right\rangle_{\mathbf{G}_0,\mathbf{F}_0} Z \, d\R, \, g\in \mathbf{G}_0.$$
\item[iv)] The market is complete and Assumption \ref{Usuposdelta} holds, so $ L_{   | \cdot |   V 	\circ 1/ | \cdot |  }$ is reflexive\item[v)]  Assumption \ref{assgnral} holds, in particular  $\frac{d\qu}{d\R} \subset L_{   | \cdot |   V 	\circ 1/ | \cdot |  }$ is    $\sigma( L_{   | \cdot |   V 	\circ 1/ | \cdot |  },   L_{U^{-1}  \circ| \cdot| })$- closed. 

\end{itemize}
\end{assumption}
\medskip

We observe that if point  (c) of Assumption \ref{assgnral}
holds and   Assumption \ref{Usuposdelta} on $U$ is  enforced, then  point  iv) of Assumption \ref{suposQLeo} is satisfied and one can write
 \begin{equation*}
   \frac{d\qu}{d\R} = \bigcap_{\lambda \in \Lambda} \left\{ \frac{d\Q}{d\R}   \, : \, \,  \frac{d\Q}{d\R} \in L_{   | \cdot |   V 	\circ 1/ | \cdot |  }  \mbox{  and } \E^{\Q} \left(h_{\lambda}
 \right)\in [a_{\lambda},\infty) \right\}
 \end{equation*}
for some family $(h_{\lambda})_{\lambda  \in \Lambda}$     of elements of $L_{U^{-1}  \circ| \cdot| }$ and  some $a=(a_{\lambda})_{\lambda  \in \Lambda} \in \RR^{\Lambda} $, by Hahn-Banach Theorem. This grants that points i) and ii) of  Assumption \ref{suposQLeo} hold with  
 $\mathbf{F}_0=\RR^{\Lambda} $, $\mathbf{G}_0=\bigoplus\limits_{\lambda \in \Lambda} \RR 
 $,    $\theta(\omega)=(h_{\lambda}(\omega))_{\lambda\in \Lambda}$ and $ \mathbf{C}_0 =\Pi_{\lambda \in \Lambda}  [a_{\lambda},\infty) $;  point iii)   holds if  the family $(h_{\lambda})_{\lambda  \in \Lambda}\cup \{1\}$    is linearly independent, or otherwise can be obtained by replacing $\mathbf{G}_0$ by a suitable quotient space (see Section \ref{worstcasecomplete} for details or the Example \ref{weak insider} below  for a concrete instance). 
 Assumption \ref{suposQLeo}  is  not an actual  restriction,  in the setting of Theorem \ref{minimaxsinreflexividad} in  complete markets. In turn,  it allows us to deal with  uncertainty sets  naturally arising in modeling situations, for instance  market information  specified by  moments of observables,  the probability of a given event, or even the full law of a given observable or  the flow of time-marginal laws of a random process  (as considered  e.g. in \cite{Bau}).  See Examples \ref{BSRobust} to \ref{weakflow}  below in this section. 

In order to characterize the minimizing density  $\hat{Z}=\frac{d \hat{Q}}{d\Prob}$ in  \eqref{dualcomplete} and following \cite{Leo-ent,LeoAbstract,Minvery},  we will formulate a dual  problem to it  in some  space $ \mathbf{G}$,  solvable under  some weak qualification condition.
 To that end notice that, as a consequence of points i),ii) and iii) of Assumption  \ref{suposQLeo} (see  Section \ref{worstcasecomplete} for additional discussion), the mapping
$$( g_0,a)\in   \mathbf{G}_0\times \RR \mapsto    \langle g_0,\theta (\cdot)\rangle_{ \mathbf{G}_0, \mathbf{F}_0}+a  \in  L_{U^{-1}\circ| \cdot|} $$
embeds  the space $ \mathbf{G}_1 :=    \mathbf{G}_0\times \RR$ into  $ L_{U^{-1}\circ| \cdot|} $ and thus  induces a norm on $ \mathbf{G}_1$.  We denote by $ \mathbf{G}$ the completion of $ \mathbf{G}_1$
under it, which is isomorphic to a closed subspace of $L_{U^{-1}\circ| \cdot|}$, and call 
 $ \mathbf{F}$ the topological dual of $ \mathbf{G}$, which is a linear subspace of $ \mathbf{F}_1 :=    \mathbf{F}_0\times \RR$. Write  $\langle g, f \rangle$ for  the natural dual product of the pair $(g,f) \in \mathbf{G}\times \mathbf{F}$ and denote by $$\langle g, \theta_1 \rangle\in  L_{U^{-1}\circ| \cdot|} $$ the element identified with $ g\in \mathbf{G}$; in particular, $\langle g, \theta_1 \rangle= \langle  g_0  , \theta  \rangle_{\mathbf{G}_0,\mathbf{F}_0}  + \beta\cdot 1 $ if    $g= (g_0,\beta) \in \mathbf{G}_1$.  Setting $ \mathbf{C}_1:= (\mathbf{C}_0\times\{1\} )$ and $ \mathbf{C}:= \mathbf{C}_1 \cap\mathbf{F}$, the dual problem of  \eqref{dualcomplete} is:
\begin{equation}\label{dualdualcomplete}
\textit{Maximize } \inf_{f \in \mathbf{C}} \langle g,f  \rangle - y 
 \E\left[ U^{-1} \left( \left( \langle g,\theta_1 \rangle\right)_+\right) \right] 
,\,  g \in \mathbf{G}. 
 \end{equation} 
\noindent  The following functional 
   \begin{equation}\label{Gamma*}
  \Gamma_y^*(f,s):= \sup_{g\in \mathbf{G}_0} \sup_{\beta \in\RR}  \,  \langle  g  , f  \rangle_{\mathbf{G}_0,\mathbf{F}_0}  + s\beta  - y   \E\left[ U^{-1} \left( \left( \beta+   \langle g ,\theta\rangle_{\mathbf{G}_0,\mathbf{F}_0}\right)_+\right) \right]   \, ,  \quad( f,s)\in  {\mathbf{F}_1} 
  \end{equation}
  will be useful  to state sufficient conditions for  primal-dual equality between the pair of problems \eqref{dualcomplete} and \eqref{dualdualcomplete}. Moreover, we will also state in terms of it  a weak qualification  condition ensuring 
 dual attainability and allowing us to characterize  the solution of \eqref{dualcomplete}. 
 Recall that the \textit{affine hull} aff$(A)$ of $A \subset L$, where $L$ is a linear space, is the smallest affine subspace of $L$ containing $A$, and the  {\it intrinsic core}  of $A$,   given by  $$\mbox{icor}(A):=\left\{ a \in A| \forall x \in \mbox{aff}(A),\exists t>0 \mbox{ st. } a+t(x-a)\in A \right\},$$
is the largest topology-free notion of  its \textit{interior}.  
Our main result in the complete case is: 
\smallskip

\begin{theorem}
\label{teocompletocaractintro}
Suppose  that Assumption \ref{suposQLeo}  holds. 
\begin{itemize} 
\item[a)]  For each $y>0$, the following identities hold:
\begin{equation}\label{dualLeointro}
  \begin{split}
  v(y)=  & \inf_{f\in \mathbf{C}_0} \Gamma_y^*(f,1)\\
= &  \sup_{g\in \mathbf{G}_0} \sup_{\beta \in\RR}  \left( \inf_{f\in \mathbf{C}_0}     \langle  g  , f  \rangle_{\mathbf{G}_0,\mathbf{F}_0}  + \beta \right) - y   \E\left[ U^{-1} \left( \left( \beta+   \langle g ,\theta\rangle_{\mathbf{G}_0,\mathbf{F}_0}\right)_+\right) \right]   \\
= &  \sup_{g\in \mathbf{G}} \left( \inf_{f\in \mathbf{C}}     \langle  g  , f  \rangle\right) - y   \left[ U^{-1} \left( \left( \langle g,\theta_1 \rangle\right)_+\right) \right] 
. \\
\end{split}
\end{equation}
Moreover,  if   $\mathbf{C}_1\cap \mbox{ dom} (\Gamma^*_{y})  \neq \emptyset$ then the infimum  in \eqref{dualcomplete}
 is attained at a unique element $Z^y \in \frac{d\qu}{d\R}$ and the  four  expressions in \eqref{dualLeointro}  equal  $ \E (\gamma^*_y(Z^y)) $.  If in addition  $\mathbf{C}_1\cap \mbox{icor (dom} (\Gamma^*_{y}))  \neq \emptyset$, 
then problem \eqref{dualdualcomplete} has a solution  $g\in \mathbf{G}$.
   \item[b)]  
    A pair $( Z^y,g^y) \in L_{   | \cdot |   V 	\circ 1/ | \cdot |  } \times \mathbf{G} $   solves problems \eqref{dualcomplete}
  and   \eqref{dualdualcomplete}
if and only if   
\begin{equation}
 \begin{cases} 
 \bullet  &   ( \Theta\left( Z ^y\right),1) \in  \mathbf{C} \cap  dom \,  \Gamma^*_y \, , \\
\bullet  &  \langle  g^y,   ( \Theta\left( Z^y \right),1)  \rangle \leq     \langle  g^y , f\rangle \, \mbox{  for all }   f \in  \mathbf{C}\cap  dom \,  \Gamma^*_y  \mbox{ and }\\
\bullet & Z^y   = y \left[ U^{-1} \right] ' \left( \left( \langle g^y,\theta_1 \rangle\right)_+\right) 
. \\
 \end{cases}
 \end{equation}
 In particular, $\left( \langle g^y ,\theta_1 \rangle\right)_+=\left( \langle \hat{ g}^y,\theta_1 \rangle\right)_+$,  $\Prob-$a.s. for any  solutions  $g^y,\hat{g}^y \in \mathbf{G}$ to   \eqref{dualdualcomplete}. 
 
  \item[c)] If    $\mathbf{C}_1\cap \mbox{ icor(dom} (\Gamma^*_{y}))  \neq \emptyset$ for all $y>0$, then
   for all $x>0$,  we have:
\begin{equation*}
u(x)=  \inf_{y>0} \left(\inf_{f\in \mathbf{C}_0} \Gamma_y^*(f,1) +xy \right) 
=  \inf_{y>0} \left(  \E\left[Z^yV\left(\frac{y}{Z^y} \right)\right]  +xy \right)
=  \E\left[Z^{\hat{y}}V\left(\frac{\hat{y}}{Z^{\hat{y}}} \right)\right]  +x{\hat{y}},
\end{equation*}
where  $\hat{y}$ belongs to the super-differential of $u$ at $x$.
    \end{itemize}
\end{theorem}

\smallskip 
Thus,  in a complete market and under the assumptions of Theorem \ref{teocompletocaractintro}, 
finding the worst-case measure $\hat{\Q}$ attaining the infimma  in \eqref{igualdades} amounts to first  finding for  each $y>0$ a solution  $g^y$ to   \eqref{dualdualcomplete} and computing  $v(y)=\E\left[Z^{y}V\left(\frac{y}{Z^{y}} \right)\right]  $, where   $Z^y =  y \left[ U^{-1} \right] ' \left( \left( \langle g^y ,\theta_1 \rangle\right)_+\right)  $, then  finding  $\hat{y}>0$ that minimizes the obtained values of $v(y)+xy$ and setting $\hat{\Q}=Z^{\hat{y}}\cdot \Prob$.
 
 \smallskip 

For each $y>0$, problem  \eqref{dualdualcomplete}   dual  to   \eqref{dualcomplete} is, in a way,  a ``dual of a dual problem'' to the original problem \eqref{RobPortOpt}. The difference is that the first dualization is w.r.t.\ the budget constraint whereas the second  one is w.r.t.\ the constraints determining the uncertainty set. The assumption $\mathbf{C}_1\cap \mbox{icor (dom} (\Gamma^*_{y}))  \neq \emptyset$  corresponds to a constraint qualification condition of geometric (rather than topological) type for the last dualization.  
 Note that  in  many practical instances,   problem   \eqref{dualdualcomplete}  can  be finite-dimensional: 
 
 \smallskip 

  \begin{example}\label{BSRobust}
  Consider  $S_t=\exp \left\{ -\frac{\sigma^2}{2}t + \sigma W_t \right\}$  the risk-neutral Samuelson-Black-Scholes   model, with  $W$  a  standard Brownian motion, $\sigma^2>0$ and $S_0=1$ (for simplicity). We take $U(x)=2 x^{1/2}  $ and   $\qu_A:= \{\Q \in {\cal P}: \Q \ll \Prob,  \,
  \E^{\Q}(S_T)\geq A, d\Q/d\Prob \in L^2\} $  for fixed $A>0$,  
   so that  $L_{ | \cdot |  V 	\circ 1/ | \cdot |   }=L_{U^{-1}  \circ| \cdot| }=L^2$ and   $\frac{d \qu_A}{d\Prob} $ is weakly closed in $ L^2$,  by Example  \ref{QnotcloL0} with $h:=S_T$.  Girsanov Theorem   yields   for each $A>0$ the existence of a probability measure $\Q_A$ s.t.    $\frac{d \Q_A}{d\Prob}\in L^2$ and  $\E^{\Q_A}(S_T)=A$, hence $\qu_A\neq \emptyset$. 
  Moreover, $\qu_A$ is closed  under infinite convex combinations, and since $S_T$ and $1$ are obviously  linearly independent r.v., Assumption \ref{suposQLeo} holds. 
  
\noindent Furthermore, we can  directly  check that $\Theta_1(\frac{d \Q_A}{d\Prob})\in  dom (\Gamma^*_{y})$  (or alternatively use  the ``little dual equality'' \eqref{little dual}) in order to get that 
 $\mathbf{C}_1\cap \mbox{dom} (\Gamma^*_{y})  \neq \emptyset$.  Since  for  any $(a,b)\in \RR^2_+$ with $a,b\neq 0$  there is  $Z\in L^2, Z\geq 0$  such that $(\E(Z),\E(Z S_T))=(a,b)$ (take e.g.  $Z:= a \frac{d \Q_A}{d\Prob}\in L^2$ with  $\Q_{A}$ as above with  $A=\frac{b}{a}$) we similarly check that  $\mathbf{C}_1\cap \mbox{icor(dom} (\Gamma^*_{y}))  \neq \emptyset$. 
 
    We next solve the (second) maximization problem in   \eqref{dualLeointro}, that is 
  \begin{equation} \notag
\begin{split}
  \sup_{(\beta,\alpha)} \left[\inf_{c\geq A} 
\beta+c\alpha- \E^{\Prob}\left(yU^{-1}(\beta+S_T\alpha)_+ \right) \right] = &  \sup_{\beta\in \RR,\alpha \geq 0} \left[\beta+A\alpha- \E^{\Prob}\left(yU^{-1}(\beta+S_T\alpha)_+ \right) \right]\\
= & \sup_{\beta\in\RR,\alpha \geq 0} \beta+A\alpha- \frac{y}{4}\E^{\Prob}\left((\beta+S_T\alpha)^2 \ind{\beta+S_T\alpha > 0} \right) \\
\end{split}
\end{equation}
 In order to get explicit expressions, we assume that $e^{\sigma^2 T}>A>1$. Upon explicitly computing the  expectation, we   notice that, in that case, the unique critical point of the above concave function  is  $g^y=\left(\frac{2(e^{\sigma^2 T}-A)}{y\left(e^{\sigma^2 T}-1 \right)}, \frac{2(A-1)}{y\left(e^{\sigma^2 T}-1 \right)}\right) \in (0,\infty)^2$, with optimal value  $\frac{1}{y}\left[1+\frac{(A-1)^2}{e^{\sigma^2 T}-1}\right]$ (see Example 2.2.3 in \cite{JBtesis} for details). We deduce that
 \begin{align*}
u(x)&=2\sqrt{x\left(1+\frac{(A-1)^2}{e^{\sigma^2T}-1} \right)}  \quad  , &
\hat{\Q}(d\omega)&:=\frac{e^{\sigma^2 T}-A+S_T(A-1)}{e^{\sigma^2 T}-1}\Prob(d\omega).
\end{align*}
Hence $\hat{\Q}$ is the unique convex combination of the measures $\Prob$ and $S_T \cdot \Prob$ which is a probability measure and satisfies $\EE^{\hat{\Q}}(S_T)=A$.
Classic results in the non-robust setting (cf. \cite{KrSch}) yield
$$\hat{X}_T:= x
\frac{\left( e^{\sigma^2 T}-A+S_T(A-1)\right) ^2}{\left(e^{\sigma^2 T}-1+ (A-1)^2\right) \left(e^{\sigma^2 T}-1\right)} \, , \,  \Prob \mbox{  and }\hat{\Q} \mbox{ a.s. } $$
and the robust optimal strategy can then be  derived by standard hedging arguments.

\end{example}

\smallskip 

 \begin{example}\label{weak insider}
 Let $(E,\Sigma)$ be measurable  space and $\vartheta:\Omega\to E$  a measurable ``observable'' of the market. Let $\nu \ll \mu:= \Prob\circ \vartheta^{-1}$  be   a probability measure on $E$ with $ \frac{d\nu }{d\mu}  \in   L_{   | \cdot |   V 	\circ 1/ | \cdot |  } (E,\Sigma,\mu) $ and assume 
$$  \frac{d\qu}{d\R} =\left\{Z  \in   L_{   | \cdot |   V 	\circ 1/ | \cdot |  }\, : \EE(Z)=1, \, Z\geq 0  \mbox{ a.s. and   }(Z\cdot \Prob )\circ \vartheta^{-1}= \nu  \right\}.$$
 Taking 
 $\mathbf{G}_0={\cal B}\slash {\cal B}^{\Prob}$ with  ${\cal B}:=\{ g: E\to \RR ,  \mbox{ bounded measurable}\}$ and ${\cal B}^{\Prob}:=\{ g\in{\cal B}: g=cst. \,\mu-  a.s.\}$, 
 $\mathbf{F}_0=\{ f:\Sigma \to \RR:  f \mbox{  finite signed measure}\}$ with  $ \langle g+ {\cal B}^{\Prob} ,f\rangle_{\mathbf{G}_0,\mathbf{F}_0} : =\int g(x)f(dx)$ and  $\theta(\omega)=\delta_{\vartheta(\omega)}$,  points i) to  iv) of Assumption \ref{suposQLeo}  hold. Also,    problem \eqref{dualdualcomplete} is   equivalent to 
  \begin{equation*}
\textit{Maximize }
 \E\left[ g (\vartheta) \frac{d\nu }{d\mu} (\vartheta)- y U^{-1} \left( g_+ (\vartheta)  \right) \right] 
 ,\,  g \in L_{ U^{-1}  \circ| \cdot|}( E,\Sigma,\mu), 
 \end{equation*} 
The first order optimality condition  for this problem is $$ \E\left[ g (\vartheta) \left(\frac{d\nu }{d\mu} (\vartheta)- y \left[ U^{-1} \right] ' \left(  g_+^{y,U}(\vartheta) \right)  \right) \right] =0,$$ for all bounded measurable $g:E\to \RR$. From this and part b) of Theorem \ref{teocompletocaractintro} we get that, provided we can always find $g^{y,U}$ such that $y \left[ U^{-1} \right] ' \left(  g_+^{y,U}(\vartheta) \right)=\frac{d\nu }{d\mu} (\vartheta) , \, \Prob- $a.s.\ , then the primal solution is $Z$ independent on $y$ and $U$, and so we recover the least-favourable measure found in \cite{Bau}. It is clear we can indeed find such $g^{y,U}\in \mathbf{G}$ in this case.\end{example}

\smallskip 

\begin{remark} Example \ref{weak insider}  points out to a more general result. Indeed, it is not difficult to see that if $\mathbf{C}=\{f\}$ is a singleton, which by Hanh-Banach can be associated to a unique minimal representative $Z\in L_{   | \cdot |   V 	\circ 1/ | \cdot |  }  $ characterized by   $\langle g,f\rangle = \EE[Z\langle g,\theta_1\rangle]$  and its  measurability w.r.t.\  the sigma-field generated by $ \{\langle g,\theta_1\rangle :g\in \mathbf{G}\}$ (see the beginning of Section \ref{charactmin}), then provided $Z=y[U^{-1}]'(\langle g^{y,U},\theta_1\rangle_+)$ is always solvable we get as before that $Z$ is the  primal solution and this is independent of $y$ and $U$. We roughly conjecture in the complete case that  the existence of a least-favourable measure in $\qu$ is related to the properties that, for every nice Young function $\phi^*$ such that the associated Orlicz space is reflexive and contains $d\qu/d\Prob$, the set $\qu$ can be written down with $\mathbf{C}=\{f\}\subset \mathbf{F}$  a singleton and that the minimal representative $Z$  satisfies $[\phi ']^{-1}(Z)\in \{\langle g,\theta_1 \rangle_+:g\in \mathbf{G}\}$.

\end{remark}

\smallskip 

\begin{example}\label{weakflow}
Assume $\vartheta=(\vartheta_t)_{t\in [0,T]}$ is under $\Prob$  a continuous process with values in $\RR^d$ and 
$$  \frac{d\qu}{d\R} =\left\{Z  \in   L_{   | \cdot |   V 	\circ 1/ | \cdot |  }\, : \EE(Z)=1, \, Z\geq 0  \mbox{ a.s. and   }(Z\cdot \Prob )\circ \vartheta_t^{-1}= \nu_t , \, t\in [0,T]  \right\}$$ 
for a flow of  probability laws  $(\nu_t)_{t\in [0,T]}$ s.t. $\nu_t \ll  \Prob\circ \vartheta_t^{-1}$ (as succinctly  studied  in \cite{Bau}). We can  take   $\mathbf{G}_0=\{ g\in C([0,T]\times \RR^d,\RR): \mbox{vanishing when }|x|\to \infty  \}$, $\mathbf{F}_0=C([0,T], {\cal M}(\RR^d))$, where  $ {\cal M}(\RR^d)$ is the space of finite signed measures on $\RR^d$ endowed with the weak topology, $ \langle g ,f\rangle_{\mathbf{G}_0\,\mathbf{F}_0} :=\int_0^T \int_{\RR^d}  g(t,x) f_t(dx) dt $ and $\theta= (\delta_{\vartheta_t})_{t\in [0,T]}$. The validity of Assumption  \ref{suposQLeo} and the   solvability of problem \eqref{dualdualcomplete}  will in general depend on the  market and on  $\vartheta$, and can be studied in specific instances (this is work in progress, but see \cite[Chapter 3.6.2]{BackhofftesisUChile} in Spanish).
\end{example}

\section{Orlicz-Musielak  spaces and the robust optimization problem}
\label{Orlicz}

We now introduce some general functional spaces needed in our study of the robust optimization problem. These can actually be seen as  Orlicz spaces based on ``randomized Young functions''.  Their main properties including dual spaces and reflexivity are first reviewed in Section \ref{OMtopo},  following succinctly  the presentation in \cite{KozekBan,Kozekint}. Then in Section \ref{secapli} we translate and apply these concepts to the robust optimization setting, for which some relevant functionals are introduced and a few technical results are established.  
\subsection{Orlicz-Musielak Spaces}
\label{OMtopo}
Recall  that  $(\Omega,\F,\R)$ is a (complete) probability space and that the notation $\E(\cdot)$ is employed for the expectation under  $\R$. 

\medskip

\begin{definition}
\label{rho-functional}
A functional $\rho:\RR \times \Omega \rightarrow [0,\infty]$ is said to be a rho-functional if the following hold:
\begin{enumerate}
\item $\rho$ is jointly measurable
\item for almost every $\omega\in\Omega$,  $\rho(\cdot,\omega)$ is lower-semicontinuous and convex
\item $\rho(0,\cdot)\equiv 0$ and $\rho(x,\cdot)=\rho(-x,\cdot)$
\item If $\alpha:\Omega \rightarrow (0,\infty)$ is measurable, then there exists a measurable function $\lambda:\Omega \rightarrow (0,\infty)$ such that a.s. $\left[|x|\geq \lambda(\omega) \Rightarrow \rho(x,\omega)\geq \alpha(\omega)  \right]$ .
\item If $\epsilon:\Omega \rightarrow (0,\infty)$ is measurable, then there exists a measurable function $\rho:\Omega \rightarrow (0,\infty)$ such that a.s. $\left[|x|\leq \rho(\omega) \Rightarrow \rho(x,\omega)\leq \epsilon(\omega)  \right]$ .
\item The random variables  $\rho(x,\cdot)$ and  $\rho^*(y,\cdot):=\sup_{x\in (-\infty,\infty)}( xy- \rho(x,\cdot)) $ are integrable for every $x, y \in (-\infty,\infty)$.

\end{enumerate}
\end{definition} 

\medskip

\begin{remark}
  Under the  conditions in Definition \ref{rho-functional}, the results in \cite{KozekBan} are valid. It is worth noting that in that paper  a functional  $\rho$  satisfying conditions 1.\  through 5.\  was called an ``N-function''. However, such a $\rho$ ``only'' converges a.s.\ to zero (resp. to $\infty$) when $x$ tends to zero (resp. to $\infty$), whereas in the nowadays standard definition of N-functions, it is the quotient $\frac{\rho(x,\omega)}{x}$ that has this limiting behaviour in $x$ near $0$ and $+\infty$. To avoid confusions we use here the  different ``rho-functional''  terminology. Also,   we note that in the language of \cite{KozekBan}, the above condition 6.\  amounts to  requiring  ``condition B on  $\rho$ and $\rho^*$'', and is necessary  to obtain topological duality results.  Last, it is not difficult to see from the above definition that $\rho^*$ is also a rho-functional.
    \end{remark}

Define now  for a random variable $Z:\Omega \rightarrow (-\infty,\infty)$,
$$I_{\rho}(Z):=\E\left[ \rho( Z,\cdot)\right]\leq \infty$$

In the terminology of \cite{KozekBan}, this is a normal convex modular. This allows us to define the following spaces:

\medskip
\begin{definition}
\label{defiOrliczspaces}
The Orlicz-Musielak space associated to $\rho$ is defined as:
\begin{equation}
L_{\rho}(\Omega,\R):=\left\{Z\in L^0 \mbox{ s.t.  }  \exists \alpha >0, I_{\rho}(\alpha Z) < \infty \right\}, \label{Lrho}
\end{equation}
and its  \textit{Orlicz heart} is the subspace:
\begin{equation}
E_{\rho}(\Omega,\R):=\left\{Z\in L^0 \mbox{ s.t.  } \forall \alpha >0, I_{\rho}(\alpha Z) < \infty \right\}. \label{Erho}
\end{equation} 
\end{definition}

In the following, $L_{\rho}$ will stand as an abbreviation for $L_{\rho}(\Omega,\R)$. The following result is a compendium of known facts; see Theorem 2.3.1 in \cite{JBtesis} for the references:

\medskip
\begin{theorem}
\label{variasobsorlicz} The following functionals define equivalent norms on $L_{\rho}$:
\begin{eqnarray}
\| Z \|^l_{\rho} &:=& \inf \left\{\beta>0: I_{\rho}\left(\frac{Z}{\beta}\right)  \leq 1 \right\}, \\
\|Z\|^a_{\rho} &:=& \sup\left\{\E(\phi Z) \mbox{ : } \phi \in L_{\rho^*}, \hat{I}_{\rho}(\phi) \leq 1  \right\}\\
&=& \sup\left\{\E(\phi Z) \mbox{ : } \phi \in L_{\rho^*}, \vert\vert \phi \vert\vert_{\rho^*}^l \leq 1  \right\},
\end{eqnarray}
where $\hat{I}_{\rho}(\phi):=\sup\limits_{Z \in L_{\rho}}\left[\E(\phi Z)- I_{\rho}(Z) \right]=I_{\rho^*}$, and $\rho^*(\cdot,\omega)$ is the a.s.\ convex conjugate of $\rho(\cdot,\omega)$ as defined previously.
Moreover,  the norm $\|\cdot\|^a_{\rho}$ has the equivalent expression 
\begin{equation}
\|Z\|^a_{\rho} = \inf_{k>0}\left\{ \frac{1}{k} \left(1+ I_{\rho}(kZ) \right) \right\}.   \label{normorlicz}
\end{equation} 
Under these equivalent norms, the linear space $L_{\rho}$ is a Banach space.

Finally, when $\rho$ is finite the topological dual of $E_{\rho}$  is isometrically isomorphic to $L_{\rho^*}$ (assuming that in one space a $\| \cdot \|^l$ norm is taken and in the other a $\| \cdot \|^a$ norm is taken) with the identification $[\phi \in E_{\rho}^* \leftrightarrow g \in L_{\rho^*}] \iff [\phi(Z)=\E(Zg),\forall Z \in E_{\rho}]$. 

\end{theorem}
\medskip
The  norms $\| \cdot \|^l_{\rho}$ and $\| \cdot \|^a_{\rho}$ are called respectively Luxemburg and Amemiya norms. Now thanks to Young's inequality, one can derive a series of H\"older inequalities:
$$\E(|Zg|)\leq 2 N_{\rho}(Z)N_{\rho^*}(g)$$
where $N_{\rho}$ (resp. $N_{\rho^*}$) represents any of the norms in $L_{\rho}$ (resp. $L_{\rho^*}$) introduced in Theorem \ref{variasobsorlicz}.  In particular,   $L_{\rho^*}$ (resp.  $L_{\rho}$)  is embedded   in the topological dual of  $L_{\rho}$ (resp.  $L_{\rho}^*$), and  $L_{\rho}$  and  $L_{\rho^*}$ are  continuously embedded in $L^1$.   The following  growth property of a rho-functional  and its relation with   topological properties  of the associated Orlicz-Musielak  space is relevant:
\begin{definition}
\leavevmode 
\label{deltados}
A finite rho-functional $\rho$ is said to satisfy the $\Delta_2$ condition (or $\rho \in \Delta_2$), if there is a constant $K\geq 1$ and a non-negative integrable function $h$ such that a.s.:
\begin{equation}
\rho(2x,\omega)\leq K\rho(x,\omega)+h(\omega).
\end{equation}
\end{definition} 
We now state Corollary 1.7.4  in \cite{Kozekint} as:

\begin{theorem}\label{OMreflex}
Let $\rho$ satisfy condition $\Delta_2$. Then $E_{\rho} = dom\left(I_{\rho}\right) = L_{\rho}$ and hence $(L_{\rho})^*$ is isometrically isomorphic to $L_{\rho^*}$. 
 Moreover, if the measure $\R$ is non-atomic, the condition $\Delta_2$ is also necessary  for this last isomorphism to hold. 

Therefore, if both $\rho$ and $\rho^*$ satisfy the $\Delta_2$ condition,  the Banach spaces $L_{\rho}$ and $L_{\rho^*}$ are in topological duality and are reflexive. The converse is true if $\R$ is non-atomic.
 \end{theorem}

\subsection{Towards the robust optimization problem}
\label{secapli}

We next associate a family of Orlicz-Musielak spaces of the  previous  type with  the    robust maximization problem \eqref{RobPortOpt}. 
We recall first  some useful  and  well-known  properties  of the function $V$  in \eqref{V} (see Lemma 2.3.1 in \cite{JBtesis}).

  \begin{lemma}
\leavevmode 
\label{lemaVUG}
The function $V$ is strictly convex, l.s.c.\  finite and differentiable (on $(0,\infty)$), strictly  decreasing, strictly positive, and satisfies:
$$ 
\lim_{x \rightarrow \infty} \frac{V(x)}{x} =  \inf \{x: U(x)> -\infty\}\,\,\,\, \,\,\,\,\,\,\,\, \mbox{ and }\,\,\,\,\,\,\,\,\,\,\,\,
V(0)=  \lim_{x \rightarrow \infty} U(x). $$
Moreover, if  $U$ satisfies $AE(U)<1$, then condition \eqref{Vdelta} holds for $V$. 
\end{lemma}

\medskip

The next functions,  briefly introduced in Section \ref{maincomplete}, will play a central role in the sequel:

\medskip
\begin{definition}\label{gammal*}
\leavevmode 
\label{gamma*} 
For   $y \geq 0$ we define the function
\begin{eqnarray} 
\gamma^*_y(z) &=& \left\{
\begin{array}{ll}
\infty & \mbox{ if } z<0,\\
zV\left(\frac{y}{z}\right) & \mbox{ if } z\geq 0,
\end{array}
\right.
\end{eqnarray}
and call $\gamma_y$ its convex conjugate. We use the convention $\frac{0}{0}=0$ to define $\gamma^*_0$.
\end{definition}

In robust optimization (a branch within optimization theory) one would call  $\gamma^*_l$ \textit{the adjoint} of $V$ (see e.g. \cite{Ben-Tal}). The next result is known, except for the third item; see the Appendix for a proof.

\medskip

\begin{lemma}
\leavevmode 
\label{lemagama}
Under point 1.\ in Assumption \ref{assgnral}, we have
\begin{itemize}
\item The function $(y,z)\mapsto \gamma^*_y(z)$ is convex on $[0,\infty)^2$. 
\item The function $\gamma^*_y(\cdot)$ is l.s.c, convex in its domain (strictly if $y>0$), on the positive half-line is increasing, finite and strictly positive, and we have $\gamma^*_y(0)=0$ and $\lim_{t\rightarrow + \infty} \frac{\gamma_{y}^*(t)}{t} = +\infty$. 
\item If $y>0$ then $\gamma_y^*(|\cdot|)=|\cdot|V\left(\frac{y}{|\cdot|}\right)$ and ${\gamma}_y(|\cdot|)=y U^{-1}(\vert\cdot\vert)$ are convex conjugates. \end{itemize}

\end{lemma}

The advantage of working with $\gamma_y^*(|\cdot|)$ is that it is a finite, even function. Notice that we have seen the functions $\gamma_y^*(|\cdot|)$ and ${\gamma}_y(|\cdot|)$, with $y=Y(\omega)>0$ in Section \ref{motivgen}. This motivates
 \begin{definition} \label{defietas} Let $Y \in \mathcal{Y}_{\R}(1)$. We denote by 
 $\eta^*_Y,\eta_Y :\RR\times \Omega \rightarrow [0,\infty]$   the functionals: 
$$\eta_Y^*(z,\omega) := \gamma_{Y_T(\omega)}^*(|z|)= |z| V\left(\frac{Y_T (\omega)}{|z|}\right) \,\,\,\,\,\,\mbox{ and } \,\,\,\,\,\,\eta_Y(z,\omega) := \gamma_{Y_T(\omega)}(|z|)=Y_T(\omega) U^{-1}(|z|).$$
 \end{definition}

Of course, if $Y_T>0$ a.s., $  \eta_Y^*(\cdot,\omega)$  and  $ \eta_Y(\cdot ,\omega)$ almost surely inherit  the obvious properties  of  $\gamma_y^*(|\cdot|) $ and $\gamma_y(|\cdot|) $  (stated e.g.\ in Lemma 2.3.3 of \cite{JBtesis}). As it is next proved,  under mild assumptions they  induce rho-functionals:

\medskip

\begin{proposition} \label{quienes}
Let $Y\in\mathcal{Y}_{\R}(1)$ be strictly positive and suppose Assumption \ref{assgnral} point 1.
\begin{itemize}
\item[a)]  Then the a.s.\  convex conjugate of the function $ \eta_Y^*(\cdot,\omega) $ is $\eta_Y(\cdot,\omega)$ and, provided that
$$\forall \beta>0, \E[V(\beta Y_T)]<\infty ,$$
 $\eta_Y^*(\cdot,,\omega) $ and $\eta_Y(\cdot,,\omega)$ are rho-functionals in the sense of Definition \ref{rho-functional}. 
\item[b)] If  condition \eqref{Vdelta} (resp \eqref{U-1delta}) holds, the  function  $\eta_Y^*(\cdot,,\omega) $ (resp.  $\eta_Y(\cdot,,\omega)$) is in $ \Delta_2$.
\item[c)] If $AE(U)<1$, then $\eta_Y^* \in \Delta_2$ and   the condition in a)  reduces to 
 $$\exists \beta >0,\E[V(\beta Y_T)]<\infty.$$
\end{itemize} 
\end{proposition}

\begin{proof}
The functionals $\eta_Y$ and $\eta_Y^*$ are clearly jointly measurable, and the fact that they are conjugate to each other follows from applying Lemma \ref{lemagama} almost surely. By properties of $U$ and $V$, as functions of $z$  they are a.s. l.s.c., even, null at the origin and  convergent to $0$ at $0$ and to infinity at infinity. Also, $\E[Y_TU^{-1}(c)]\leq U^{-1}(c)$ for every constant $c>0$ since $Y \in \mathcal{Y}_{\R}(1)$ satisfies $\E(Y_T)\leq 1$. Hence, $\eta_Y(c)$ is integrable. The assumption  $\E[V(\beta Y_T)]<\infty$ for every $\beta>0$ implies that  also $\eta_Y^*$ is integrable when applied to  constants. We conclude that they are rho-functionals.    
For the second point, notice that thanks to \eqref{Vdelta}, 
\begin{equation*}
\begin{split}
\eta_Y^*(2z)= & 2zV\left(\frac{Y}{2z}\right)\leq 2a\eta_Y^*(z)+2b(Y+z) \\ 
= & 2a\eta_Y^*(z)+2bY+2b z \mathbf{1}_{\{ z\geq Y/V^{-1}(1)\}} + 2b z \mathbf{1}_{\{ z< Y/V^{-1}(1)\}} \\
 \leq & 2a\eta_Y^*(z)+2bY+2b \eta_Y^*(z) + 2b Y/V^{-1}(1),\\
\end{split}
\end{equation*}
for every $z>0$,  which means that   $\eta_Y^*\in  \Delta_2$. The corresponding property for   $\eta_Y$ is direct.      The last statement c) follows from the last part of Lemma \ref{lemaVUG}.
\qquad 
\end{proof}

\medskip

Point (c) above should be compared with the comment before Corollary 6.1 in \cite{KrSch}.\\
With some abuse of notation, for $Z\in L^0$ we will write  simply 
$\eta_Y^*(Z)$ referring to  the function $\eta_Y^*(Z,\cdot):\Omega\to [0,+\infty) $ such that $\eta_Y^*(Z,\cdot)(\omega)=\eta_Y^*(Z(\omega),\omega)$. 

\bigskip

\begin{remark}
\label{dualKoz}
We deduce that, whenever $Y \in   \mathcal{Y}_{\R}(1)$ satisfies $Y_T>0$ a.s.\ and  $Y\in\Y^*$,   
$$L_{\eta_Y^*}= \left\{Z\in L^0 \mbox{ s.t. } \exists \alpha >0,
\E^{\R}\left[ \eta^*_Y \left( \alpha  Z  \right) \right]  < \infty \right\} $$
is an  Orlicz-Musielak  space. Moreover, $L_{\eta_Y^*}$ and $L_{\eta_Y}$ (defined analogously) are in separating topological duality and, by Theorem A.5 in \cite{KozekBan} or Proposition 1.5 in \cite{Kozekint}, $\E[\eta_Y^*(\cdot)]$ and  $\E[\eta_Y(\cdot)]$ are convex conjugates to each other w.r.t.\ the given duality. 
\end{remark}

\medskip

We end this section commenting that if $Y\in\Y^*$, then the topology of $L_{\eta_Y^*}$ is stronger than that of $L^1$ and that bounded sets in $L_{\eta_Y^*}$ are uniformly integrable, see  Lemma 2.3.5 of \cite{JBtesis}.

\section{Worst-case measures in complete markets}
\label{worstcasecomplete}
In this section we  prove Theorem \ref{teocompletocaractintro}. As explained at the outset of Section \ref{maincomplete}, we take  the reference measure to be the unique martingale measure, but the result can be generalized if this were not the case, at the price of dealing with random Young functions.
Upon introducing the useful notation 
\begin{equation}\label{etacasocompleto}
\eta^*(z):=\eta^*_1(z)= |z|V\left(\frac{1}{|z|}\right)=\gamma^*_1(|z|), \, z\in\RR,
\end{equation}
we recall that the   Orlicz-Musielak   space pertinent for the problem is the Orlicz space
    $L_{\eta^*}$. Finally, from Lemma \ref{lemagama}, we know that  the conjugate function of $\eta^*$ is the even function $$\eta:= \bar{\gamma}_1(\cdot)=\gamma_1(|\cdot|)=U^{-1}(\vert\cdot\vert).$$

\subsection{Characterization of the  minimizing measure}\label{charactmin}
 In what follows, Assumption  \ref{suposQLeo} is enforced. We briefly discuss in details its main consequences, and introduce some additional notation needed in this section. 
Note first that, under points i) and ii) the integral  
$$\int_{\Omega} \left\langle g, \theta(\omega) \right\rangle_{\mathbf{G}_0,\mathbf{F}_0} Z(\omega)  d\Prob (\omega),$$ is well defined  for each  $Z\in L_{\eta^*} $ and all  $g \in \mathbf{G}_0$, 
by H\"older's inequality;  it therefore defines the element of $\mathbf{F}_0=(\mathbf{G}_0)'$ denoted by   $\Theta(Z)$ in point iv).  
We write $ \mathbf{F}_1 :=  \mathbf{F}_0\times \RR, \,  \mathbf{G}_1=  \mathbf{G}_0\times \RR$ and     $\langle \cdot, \cdot\rangle_{\mathbf{G}_1,\mathbf{F}_1} $ for the obvious duality product between these spaces and set
  $$\theta_1(\omega) :=(\theta(\omega), 1)\in \mathbf{F}_1,\quad \Theta_1(Z):=\left( \int \theta Z\, d\R, \int  Z\, d\R\right)= \int \theta_1 Z\, d\R \in \mathbf{F}_1,$$  and $\mathbf{C}_1:=  \mathbf{C}_0\times \{1\}$. By points i),  ii) and iii) the adjoint $\Theta^*_1: \mathbf{G}_1\to  L_{\eta}$ of $\Theta_1:  L_{\eta^*}\to \mathbf{F}_1$ given by   $\Theta_1^*((g,a) )(\omega)=  \left\langle g, \theta(\omega) \right\rangle_{\mathbf{G}_0,\mathbf{F}_0} + a$  is a  linear injection,  $g\in  \mathbf{G}_1\mapsto \|\Theta^*_1( g) \|_{\eta}$  defines a norm and   $\mathbf{G}_1$ can be identified with  $\Theta^*_1(\mathbf{G}_1)$.  Notice that  iii) can always be assumed to hold, replacing  if needed $ \mathbf{G}_0$ by $\mathbf{G}_0	\slash  \mathbf{G}_0^{\Prob}$, with $ \mathbf{G}_0^{\Prob}:=  \{g\in \mathbf{G}_0 :  \Theta^*(g)=cst.  \, \Prob- a.s.\}$,     and $\langle \cdot, \cdot\rangle_{\mathbf{G}_0,\mathbf{F}_0} $ by the bi-linear map $(g+ \mathbf{G}_0^{\Prob},f)\mapsto \langle g,f\rangle_{\mathbf{G}_0,\mathbf{F}_0} $.

The  completion  $\mathbf{G}$ 
 of  $\mathbf{G}_1$ with respect to $\|\Theta^*_1( \cdot) \|_{\eta} $  is 
 isometrically isomorphic  to the closure  $\overline{\Theta_1^*(\mathbf{G}_1)}^{L_{\eta}}$ in $L_{\eta}$ and  $\Theta^*_1$ has a  equally denoted  isometric  extension to   $\mathbf{G}$. Recall that we write 
  \begin{equation*} \langle g, \theta_1\rangle:= \Theta_1^*(g),
 \end{equation*}
  for the element of  $\overline{\Theta_1^*(\mathbf{G}_1)}^{L_{\eta}}$ identified with   $g\in \mathbf{G}$. The topological dual  of  $\mathbf{G}$ is  
 $$\mathbf{F}:=\{f \in \mathbf{F}_1\, : \exists C_f>0  \mbox{ s.t. } |\langle g,f\rangle_{\mathbf{G}_1,\mathbf{F}_1}|\leq C_f   \|\Theta^*_1( g) \|_{\eta} \, \, \forall g\in \mathbf{G}_1 \}$$
and we use the  notation 
  $\langle\cdot , \cdot \rangle$   for   the natural  extension  of  the dual product $\langle\cdot , \cdot \rangle_{\mathbf{G}_1,\mathbf{F}_1}$      to   $\mathbf{G}\times\mathbf{F} $.
   Notice that  $ \Theta_1:L_{\eta^*} \mapsto \mathbf{F} $  is  continuous and  (by Hahn-Banach extension Theorem) surjective; $\mathbf{F}$ can thus be identified with the quotient of $ L_{\eta^*} $ by the annihilator $\left(\overline{\Theta_1^*(\mathbf{G}_1)}^{L_{\eta}}\right)^{\perp}$). One can  always choose $Z\in \Theta_1^{-1}(f)$  measurable with respect to  the sigma-field generated by  $\Theta^*_1(\mathbf{G}_1)$ (replacing $Z$ by   $ \EE( Z \vert {\cal G}) $ if needed).
 
 \medskip 
 
\noindent
The previous objects being introduced, we can now proceed to the proof
 of Theorem  \ref{teocompletocaractintro}.
 
 \noindent
 {\it Part a) :}  Observe  that functions $\gamma^*_y$, $\gamma_y$,     $\eta^*_y$   and $\eta_y$  correspond respectively  to functions $\gamma^*$, $\gamma$,  $\lambda_{\diamond}^*$ and $\lambda_{\diamond}$  in \cite{Leo-ent} (with, in the notation therein, $m(z)=0$ and $\gamma=\lambda$),   our mappings $\theta_1$ and $\Theta_1$ correspond respectively to  the mappings $\theta$ and $T_0$ therein,  and  our spaces and sets  $\mathbf{F}_1,  \mathbf{G}_1,  \mathbf{C}_1, \mathbf{F}$ and $ \mathbf{G}$ correspond respectively to ${\cal X}_0$,  ${\cal Y}_0$, $C$,  ${\cal X}$ and  ${\cal Y}$ in that work. Applying parts a) and b) of  Theorem 3.2 in  \cite{Leo-ent}, and since conditions 1) and 2) therein are ensured by our assumptions, we readily deduce the validity of   part a) of Theorem  \ref{teocompletocaractintro}, except for the attainability of problem  \eqref{dualdualcomplete}, which requires some additional analysis. 

Indeed, note that a solution to  \eqref{dualdualcomplete}  might in general  not exist,  since  the non-even function $\gamma_y^*$ does not provide  a control of  the Young function $\eta^*$ defining the space $L_{\eta^*}$. As in  \cite{Leo-ent}, we  need to introduce first a suitable  extension of   \eqref{dualdualcomplete}, which  will  always have  a solution in some abstract space  under our assumptions, and  prove that  it actually is an element of $ \mathbf{G}$, which thus solves \eqref{dualdualcomplete}. We point out  however  that the  results on the extended dual problem  in \cite{Leo-ent} do not   apply here (since our function $w\mapsto \gamma_y( (w)_-)$ vanishes) but we will still  be able to  follow the abstract   method    of  \cite{LeoAbstract} on which \cite{Leo-ent} relies and conclude similarly.
 
 Let us thus  introduce the  extension of  problem  \eqref{dualdualcomplete}.  We denote by $\widetilde{L_{\eta}}$
the algebraic dual of $L_{\eta^*}$ and  by  $\langle\cdot , \cdot \rangle$ the associated dual product. 
We also consider  the space $ \widetilde{\mathbf{G}} $ defined as  the algebraic dual of $\mathbf{F}$, and  we write  $\langle\cdot , \cdot \rangle$ for  the corresponding dual product   as well (which dual product  is meant should be clear from the context). Observe that the operator $\Theta_1: L_{\eta^*} \to \mathbf{F}$ naturally induces the extension 
 $\Theta_1^*: \widetilde{\mathbf{G}}\to \widetilde{L_{\eta}}$ of $\Theta_1^*: \mathbf{G}\to L_{\eta}$ given by
 $$\langle \Theta_1^*(g), Z\rangle = \langle g , \Theta_1(Z)\rangle \, , \quad (g,Z)\in \widetilde{\mathbf{G}}\times  L_{\eta^*}.$$
 Introduce  also the convex functions
$ \Phi_y(W):= y \int \gamma(W ) d\R    \, , \quad W\in  L_{\eta},$
\begin{equation}\label{Phi*yZ} \Phi^*_y(Z):= \int \gamma_y^*(Z) d\R =\sup_{W\in  L_{\eta}} \EE(Z W) -   y \int \gamma(W ) d\R   \, , \quad Z\in  L_{\eta^*}
\end{equation}
the last equality, thanks to Proposition \ref{quienes} a), and 
$$\overline{\Phi}_y(\zeta):=\sup_{ Z\in  L_{\eta^*} } \langle \zeta,Z\rangle -  \Phi^*_y(Z) \,,  \quad \zeta \in  \widetilde{L_{\eta}} \, \,. $$
 With this elements, the extended dual problem is defined as: 
 \vspace{.2cm}
\begin{equation}
\textit{Maximize } \inf_{f \in \mathbf{C} } \langle g,f  \rangle - \overline{\Phi}_y(\Theta_1^*(g)) \mbox{ , } g \in \widetilde{\mathbf{G}}. \tag{$\widetilde{D}_y$}
\end{equation} 

Recall next that a topological vector space $L$ endowed with a partial order $\leq $ is  called a Riesz space if  $\leq $ is a lattice:  $\forall \, \ell_1,\ell_s\in L$, $\exists \, \ell_1\vee \ell_2\in L$ such that $\ell_1\vee \ell_2\geq \ell_1,  \ell_2$ and   $\ell_1\vee \ell_2\leq \ell $ $\forall \ell\in L$ such that $\ell\geq \ell_i, i=1,2$.  Given $\ell\in L$, the elements $\ell_+,\ell_-$ and $|\ell|$ are then  defined  in a similar way as in $\RR$.  A dual order also written $\leq $ is induced in the algebraic dual  $L'$  of $L$. By Riesz' Theorem,  the space $L^b:=\{ \zeta\in L' :  
    \sup\limits_{\ell' \in  L,\, | \ell'| \leq \ell} |\langle \zeta , \ell' \rangle|<\infty \, \forall  \ell\in  L, \ell \geq 0\}$ of ``relatively bounded linear forms'', or 
 order dual   of $L$,   is a Riesz space too. In particular,  $ \zeta\in L^b$  admits a unique decomposition   $\zeta=\zeta_+ - \zeta_-$ 
  into positive and negative parts $\zeta_+,\zeta_-  = (-\zeta)_+ \in  L^b$,   with  $\zeta_+\geq 0$  and 
$  \langle \zeta_+ , \ell\rangle := \sup\limits_{ \ell' \in  L, \, 0\leq \ell' \leq \ell} \langle \zeta , \ell' \rangle.$
A  complete Riesz space $L$ endowed with a norm $\| \cdot\|$  such that $|\ell_1|\leq |\ell_2|\Rightarrow  \|\ell_1\|\leq \|\ell_2\|$ is called a Banach lattice, and  its order dual $L^b$ and topological dual $L^*$ coincide. We refer the reader to  \cite{Aliprantis}  Ch. 8 and 9 for  these facts and  background on Riesz spaces. 

\smallskip 
The remainder statement in  part a) of Theorem  \ref{teocompletocaractintro}, i.e.  the existence of a solution to   \eqref{dualdualcomplete}  will  follow from the two next results:

\smallskip

 \begin{lemma}\label{existdualextend}
Suppose   Assumption \ref{suposQLeo}  holds and   that
 $\mathbf{C}_1\cap \mbox{icor dom} (\Gamma^*_{y})  \neq \emptyset$.  Then,  the extended dual problem $\widetilde{D}_y$ has  a  solution. 
  \end{lemma}

\begin{proof} Existence follows applying Theorem   5.3 in \cite{Leo-ent} to  ${\cal U}=L_{\eta}={\cal U}''$, ${\cal L}=L_{\eta^*}$, ${\cal X}= \mathbf{F}$ and ${\cal Y}= \mathbf{G}$  with   our functions $\Phi_y$ and $\Theta_1$ in  the respective roles of functions  $\Phi_0$ and $T_0$ therein (notice that we have interchanged here the roles of the symbols $'$ and $^*$, used  therein to respectively  denote topological or algebraic dual spaces).  
\end{proof}

\smallskip

 \begin{lemma}\label{domPhi} Let  $\zeta\in \widetilde{L_{\eta}}$   be such that $  \overline{\Phi}_y(\zeta)<\infty.$ Then,  $\zeta$ belongs to the order dual of $Z\in L_{\eta^*}$; in particular there exists $ W^{ \zeta }  \in L_{\eta}$ such that $\langle \zeta , Z\rangle = \EE ( W^{ \zeta } Z)$ for all $Z\in L_{\eta^*}$. 
  Moreover, we have
\begin{equation}\label{phiphiphi}\overline{\Phi}_y (\zeta)=   \overline{\Phi}_y (\zeta_+)=\overline{\Phi}_{y,+}(\zeta)= \Phi_y ( (W^{ \zeta })_+ )
\end{equation}
 where  for $\zeta \in  \widetilde{L_{\eta}} $ we define
 $
 \overline{\Phi}_{y,+}(\zeta):=\sup_{ Z\in  L_{\eta^*} } \langle \zeta,Z\rangle -  \Phi^*_{y,+}(Z)$, with $ \Phi^*_{y,+}(Z) := \int \gamma_y^*(|Z| ) d\R.$
  \end{lemma}
  The proof of Lemma \ref{domPhi}  relies on Proposition 5.10 in  \cite{Leo-ent} and is given in the Appendix.  

\smallskip

We can now finish the proof of  part a) of Theorem  \ref{teocompletocaractintro}. Indeed,  Lemma \ref{existdualextend} ensures the existence of a solution $\tilde g\in \widetilde{\mathbf{G}}$ to  $\widetilde{D}_y$ which, thanks to  Lemma \ref{domPhi}, is such that   $\Theta_1^*(\tilde {g})\in  L_{\eta}$. By Theorem  5.7, a) in \cite{Leo-ent} (taking  there ${\cal X}=\mathbf{F}$, ${\cal X}^*=\tilde{\mathbf{G}}$  and  $\Lambda:= \bar{\Phi}_y\circ \Theta_1^*: \mathbf{G} \to [0,\infty]$), for some net $\{g_{\alpha}\}\subset  \mathbf{G}$ such that $ \bar{\Phi}_y(\Theta_1^*(g_{\alpha}))<\infty$ 
one has $\langle  g_{\alpha} , \Theta_1(Z)\rangle \to \langle  \tilde g, \Theta_1(Z)\rangle$  for all $Z\in L_{\eta^*}$.  In other words, $\Theta_1^*( g_{\alpha})\in L_{\eta} $ converges $\sigma(L_{\eta},L_{\eta^*})-$weakly  to  $\Theta_1^*(\tilde {g})\in  L_{\eta}$. The  set 
 $\Theta_1^*(\mathbf{G}_1)$ being convex, its  $\sigma(L_{\eta},L_{\eta^*})-$weak closure and its norm closure in $L_{\eta}$ coincide. The previous  and the  definition of $\mathbf{G}$ thus imply that $\tilde {g} \in \mathbf{G}$, hence  $\tilde {g}$  solves  problem \eqref{dualdualcomplete} too.

{\it Part b):}  We use  Theorem 5.4 in \cite{Leo-ent} stating that, in the present context,  a pair $(Z,g)\in L_{\eta^*}\times \widetilde{\mathbf{G}}$ is such that $Z$ solves   \eqref{dualcomplete} and $g$ solves $ \widetilde{D}_y$ if and only if the following hold: 
 \begin{equation}\label{charactabst}
  \begin{cases} 
 \bullet  &    \Theta_1\left(Z \right) \in  \mathbf{C}  \\
\bullet  &  \langle   \Theta_1^* (  g) ,  Z  \rangle_{\widetilde{L_{\eta}}, L_{\eta^* }}  \leq     \langle  \Theta_1^* (  g)  , Z ' \rangle_{\widetilde{L_{\eta}}, L_{\eta^* } } \, \mbox{  for all }   Z'    \in   \mbox{ dom }   \Phi^*_y   \mbox{  such that }  \Theta_1\left(Z' \right) \in \mathbf{C} \\
\bullet &  Z\in \partial_{L_{\eta^*}} \overline{\Phi_y} ( \Theta_1^* (  g)). 
 \end{cases}
 \end{equation}
 Note that, since  $\gamma(-|\cdot|)=0$, by part c) of Proposition 5.10 in  \cite{Leo-ent}   the third point is  always equivalent to
 $ Z\geq 0 $ and $ Z\in \partial_{L_{\eta^*}} \overline{\Phi}_{y,+} ( \Theta_1^* ( g))$. 
 
 Now, assume that $(Z,g) \in L_{\eta^*}\times \mathbf{G}$ solve \eqref{dualcomplete}
  and   \eqref{dualdualcomplete}.  We  have from \eqref{phiphiphi} that   $Z\in \partial_{L_{\eta^*}} \Phi_{y} (( \langle g, \theta_1\rangle) _+)$. Observe  that  $ \Phi_y$ is  G\^ateaux  differentiable in  $L_{\eta }$ with derivative at point $W\in L_{\eta }$ given by $y \gamma' (W)\in L_{\eta^* }$, as follows by dominated convergence using  the equality $\gamma(W)+\gamma^*\circ\gamma'(W)= W \gamma'( W)$ and the bounds $\gamma(2z)-\gamma(z)\geq \gamma'(z)z$ if $z\geq 0$ (by mean value theorem and increasingness of $\gamma$) and  $(k-1)\gamma(z)+d\geq  \gamma'(z)z  $ (with the  notation in \eqref{U-1delta} and where, necessarily,  $k\geq 1$). We deduce that the third point in  b) of Theorem  \ref{teocompletocaractintro}  is satisfied.  Moreover, the space $\widetilde{L_{\eta}}$ in the second point in \eqref{charactabst} 
can be replaced by $L_{\eta}$. Using the ``little dual equality''   deduced from part a) when $C_0=\{f\}$ is  a singleton with $f\in  \mathbf{F}$: 
\begin{equation}\label{little dual}
\Gamma_y^*(f)=\{\Phi_y^*(Z): Z\in L_{\eta^*}, \Theta_1(Z)=f\} \, (=  \{ \E (\gamma^*_y(Z)): Z\in L_{\eta^*}, \Theta_1(Z)=f\} )
\end{equation}
 (or  proved in part a) of Proposition 5.7 of \cite{LeoAbstract}),  we easily obtain, with the surjectivity of $\Theta_1$  the first and second points in  part b) of Theorem  \ref{teocompletocaractintro}.  

Reciprocally, if  the pair $(Z,g) \in L_{\eta^*}\times \mathbf{G}\subset L_{\eta^*}\times \widetilde{\mathbf{G}} $  satisfies the  three points  in the statement,   in a similar way it  is seen to satisfy the conditions in \eqref{charactabst} and thus solve \eqref{dualcomplete} and  $ \widetilde{D}_y$.  Since $g\in  \mathbf{G}$, it solves  \eqref{dualdualcomplete}  and the proof of part b) of Theorem  \ref{teocompletocaractintro}  is finished. 

{\it Part c):} since the assumptions of Theorem \ref{minimaxsinreflexividad}  hold true under Assumption  \ref{suposQLeo},  part  c) of Theorem  \ref{teocompletocaractintro} follows from part a) since $u$ and $v$ are  conjugate.

\section{Modular spaces and the incomplete case}
\label{incompletecase}

In this section, the robust optimization problem in the incomplete market case will be explored. Essentially the aim is to prove here Theorems \ref{minimaxsinreflexividad} and \ref{tristeintro}. In Subsection \ref{mod intro} the natural extension from the Orlicz-Musielak setting to the modular one will be motivated. Likewise the potential usefulness of this extension to the robust optimization problem will be sketched. Then in Subsection \ref{topomod} and the following one, the machinery of modular spaces and its link to the problem of robust optimization will be fully explored. The main result here is the proof of Theorem \ref {minimaxsinreflexividad}. The second crucial result is then Theorem \ref{triste}, a slight extension of Theorem \ref{tristeintro}, and the remarks thereafter. 

\subsection{Modular space associated with the incomplete case}
\label{mod intro}

Let us recall the notation :
$$\eta^*_{Y}(z)=\vert z\vert V(Y/\vert z\vert)\,\, \,\, \mbox{and}\,\, \,\,\eta_{Y}(x)= YU^{-1}(\vert x\vert),$$
of Definition \ref{defietas} and re-introduce the important functionals we already saw in Section \ref{motivgen}:
$$
\begin{array}{rcl}
I(Z):=&\inf_{Y\in\Y}\E[\eta_Y^*(Z)]&=
\hspace{0.5cm} \inf_{Y\in\Y}\E[\vert Z\vert V(Y/\vert Z\vert)],\vspace{5pt}\\ 
J(X):=& \sup_{Y\in \Y}\E[\eta_Y( X)] &=\hspace{0.5cm} \sup_{Y\in \Y}\E[YU^{-1}(\vert X\vert)].
\end{array}
$$

We start by observing that for the set
$$\Y^*=\{Y\in\Y:Y>0 \mbox{ a.s. and }\forall \beta>0, \E[V(\beta Y)]<\infty\},$$
and under Assumption \ref{assgnral}, we may compute $I$ and $J$ on $\Y^*$ simply. More exactly:\medskip

\begin{lemma}
\label{suposdensonec}
We have:
\begin{itemize}

\item[(i)]
If point 2.\ in Assumption \ref{assgnral} holds, then $$I(Z)=\inf_{Y\in \Y^*}\E[\eta_Y^*(Z)]\,\,\,\,\,\, \,\,\,\,\mbox{and}\,\,\,\,\,\, \,\,\,\,J(X)=\sup_{Y\in \Y^*}\E[\eta_Y(X)].$$

\item[(ii)]
If either the reference measure $\Prob$ is a martingale measure, or there is a continuous $\R -$local martingale $M$ and $\lambda\in L^2(M )$ such that the price process satisfies $dS_t=dM_t+\lambda_t\cdot d\langle M\rangle_t$ and $\E\left[V(\beta \mathcal{E}(-\int\lambda d M)_T)\right]<\infty$ for every $\beta > 0$, where $\mathcal{E}$ stands for the stochastic exponential, then $\Y^*\neq\emptyset$, i.e.\ point 2.\ in Assumption \ref{assgnral} holds.


\end{itemize}

\end{lemma}

From the previous lemma, we can always assume to be working with $Y\in\Y^*$ at will. The proof is given in the appendix.

\smallskip

Recall that $\eta^*_{Y}(z)=\vert z\vert V(Y/\vert z\vert)$ is a ``random Young function'' induced by $Y\in\Y$. For $Y\in\Y^*$ such functions induce a space $L_{\eta^*_Y}=\{Z\in L^0:\E[\eta^*_{Y}(\alpha Z)]<\infty, \mbox{ some }\alpha>0\}$, called Orlicz-Musielak space (see Proposition \ref{quienes}), which we will denote here $L^*_Y$ for simplicity. These spaces have, as discussed in Theorem \ref{variasobsorlicz}, several equivalent norms; for instance the Luxemburg or the Amemiya norms, respectively:
$$\| Z \|^l_Y:= \inf\{\beta>0: \E[\eta^*_{Y}(\beta Z)] \leq 1\}\,\,\, \mbox{ and }\,\,\,  \| Z \|^a_Y:=\inf_{k>0}\left[\frac{1}{k}+ \frac{\E[\eta^*_{Y}(kZ)]}{k}\right].$$
We also define the spaces $L_Y$ analogously, in terms of $\eta_Y$, the conjugate of $\eta_Y^*$.\\

It is then clear that $v(y)=y \inf_{Z\in d\qu_e/d\R} I(Z/y)$. On the other hand, recall that the function $(Y,Z)\in (L^0)_+\times (L^0)_+ \mapsto \E[ZV(Y/Z)] $ is jointly convex (as $(y,z)\rightarrow zV(y/z)$ is so) and jointly lower-semicontinuous w.r.t.\ convergence in probability (see the proof of Lemma 3.7 in \cite{SchiedWu}). Also recall the following Komlos-type argument (see Lemma A.1.1 in \cite{DelSch}): if $\{A_n\}_n$ is a sequence of positive random variables bounded in $L^0$, then there is a positive finite r.v. $A$ and a sequence $B_n\in conv\{A_n,A_{n+1},\dots\}$ such that $B_n\rightarrow A$ in probability.

We associate to the functional $I$ a set, in complete analogy to Orlicz-Musielak spaces:
\begin{equation}
L_I := \left\{Z\in L^0(\Prob): I(\alpha Z)<\infty \mbox{ for some }\alpha>0\right\},\label{espaciosLJ}
\end{equation}
and define $L_J$ accordingly in terms of $J$. Now we collect some elementary observations. The reader should notice that these spaces coincide with the ones given in Section \ref{motivgen}.
\medskip

\begin{lemma}
\label{Iconvex}
The following hold:
\begin{itemize}\item  The functionals $I,J:(L^0)_+\rightarrow [0,\infty]$ are convex and moreover $I$ is lower- semicontinuous w.r.t.\ convergence in measure. Also, for each non-vanishing $Z\in dom(I)$, the infimum in $I(Z)$ is attained at some $Y\in\Y$.
\item The set $L_I$ is a linear space coinciding with $\cup_{Y\in\Y}L^*_Y$, whereas the set $L_J$ is a linear space contained in $\cap_{Y\in\Y}L_Y$. 
\item $J(M)\leq x\iff U^{-1}(|M|)\leq X_T \mbox{ for some }X\in\mathcal{X}(x)$.
\end{itemize} 
\end{lemma}

\medskip

\begin{proof}
For the convexity of $I$, recall that the partial infimum of every jointly convex function is convex. The fact that $I(Z)$ is attained is a consequence of the closedness and convexity of $\Y$, a Komlos-type argument and  the lower semicontinuity of $Y\mapsto\E[ZV(Y/Z)]$. This in turn implies the lower semicontinuity of $I$, now because $(Y,Z)\mapsto\E[ZV(Y/Z)]$ is l.s.c.  That $J$ is convex is a consequence of the convexity of $U^{-1}$.  The equality of the  sets mentioned  in the second point is evident from the fact that for $Z$ fixed the infimum over the $Y\in\Y$ is attained. The linearity of $L_I$ follows now from the convexity of $I$: if $I(\alpha Z),I(\beta X)<\infty$, taking $\gamma=\frac{\alpha\beta}{\alpha+\beta}$ yields $I(\gamma [Z+X])=I\left[\frac{\beta}{\alpha+\beta}[\alpha Z] + \frac{\alpha}{\alpha+\beta}[\beta X] \right]\leq \frac{\beta}{\alpha+\beta}I(\alpha Z)+\frac{\alpha}{\alpha+\beta}I(\beta X)<\infty$. The linearity of $L_J$ is proved as in the case of $L_I$. It is clear that if $X\in L_J$ then also $X\in L_Y$, for every $Y\in \Y$.
The last point goes by definition of $J$ and Proposition 3.1.ii in \cite{KrSch}.
\qquad 
\end{proof}

We shall see in the next section that $\vert Z \vert^a_I = \inf_{k>0}[\frac{1}{k} + \frac{I(kZ)}{k}]$ is a norm on $L_I$, making it a Banach space. Further this norm-topology will be stronger than that of convergence in measure. This implies immediately that $I$ will be lower-semicontinuous with respect to $\vert \cdot \vert^a_I$. In light of this, let us justify the appeal of the space $L_I$:

\smallskip 
\begin{remark} \label{ilusion}
Since $v(y)=y\inf_{Z\in d \qu_e/d\R} I(Z/y)$, and also by definition $\vert Z \vert^a_I \leq y+yI(Z/y)$, by taking a minimizing sequence $\{Z_n\}$ such that $yI(Z_n/y)$ decreases to $v(y)$ it follows that the sequence $\{Z_n\}$ would be bounded in $(L_I,\vert \cdot \vert^a_I)$.  On the other hand, we shall see in Proposition \ref{coercividad2} that $u_{\Q}(x) \geq c\vert Z \vert^a_I$. This shows that in minimizing the $u_{\Q}$'s we may restrict $\qu$ to its intersection with a given ball. Hence the two previous estimates show that requiring $d\qu/d\R$ to be closed in $(L_I,\vert \cdot \vert^a_I)$ and asking for conditions on the ingredients of the problem so that this space becomes reflexive, would allow to fully solve the robust optimization problem. We will see, however, that $L_I$ is reflexive \textit{almost exactly} when the market is complete, and that this is independent of how well-behaved our utility function is (in stark contrast to the complete case). On the other hand, because we will be able to prove that $L_J$ is a norm-dual space, and since the image through $U$ of the terminal wealths live in this space and are unifomly norm-bounded, we can still derive a minimax identity.
\end{remark}

\subsection{Modular spaces $L_F$ and $E_F$; topological/duality results}
\label{topomod}

Generating a space from a functional is a classical subject. See e.g. \cite{Nakano,Musielak}. There are quite minimalistic conditions ensuring that the generated space be an \textit{F-space} and that some related functionals form a family of pseudo-norms for it. Here, rather than working at this level of generality, a more relaxed terminology and a lighter approach (as in chapter XI in \cite{Nakano}) will be pursued.\\
We first introduce the notion of a convex modular, and then its associated modular space. We shall see that $I$ (respect.\ $J$) and $L_I$ (respect.\ $L_J$) fulfil these definitions.
\begin{definition}

A functional $F:{\cal S}\rightarrow [0,\infty]$ over a vector space ${\cal S}$  is called a \textit{Convex modular}  if the following axioms are fulfilled:

\begin{enumerate}
\item $F(0)=0$
\item $F(s)=F(-s)$
\item $\forall s\in {\cal S},\exists \lambda>0:F(\lambda s)<\infty$
\item $F(\xi s)=0$ for every $\xi>0$ implies $s=0$
\item $F$ is convex
\item $F(s)=\sup\limits_{0\leq \xi < 1} F(\xi s)$
\end{enumerate}

\end{definition}

With this definition, it follows that on the space:
$$L_F({\cal S}):=\{s\in {\cal S}: \lim_{\alpha\rightarrow 0} F(\alpha s)=0 \}= \{s\in {\cal S}: F(\alpha s)<\infty \mbox{ for some }\alpha>0\}$$
the following functionals are equivalent norms, called respectively Luxemburg and Amemiya norms:
$$\vert s \vert_F^l = \inf\{\beta>0: F(s/\beta)\leq 1 \} \mbox{ and  } \vert s \vert_F^a = \inf\left\{ \frac{1}{k}+ \frac{F(ks)}{k} :k>0\right\},$$
and actually thanks to Theorem 1.10 in \cite{Musielak}, $\vert s \vert_F^l \leq \vert s \vert_F^a\leq 2\vert s \vert_F^l$. It can be proved, as in chapter XI, 81 in \cite{Nakano}, that the topology induced by the Luxemburg norm is exactly the (weakest locally convex topology) generated by the family of neighbourhoods of the origin $\{F^{-1}(-\infty,c])\}_{c}$. The space $L_F$ is called a \textit{modular space associated to $F$}.

Now recalling the definitions in the previous subsection, we prove:\\


\begin{proposition}
\label{modfunspace}
The functional $I$ is a convex modular and $L_I$ is a modular space associated to it. Likewise, $J$ is a convex modular and $L_J$ is a modular space associated to it.
\end{proposition}
\begin{proof}
For $I$ first. Axioms ($1$), ($2$) and ($3$) hold by definition, and ($5$) is proved in Lemma \ref{Iconvex}. For ($4$) notice that $I(\xi Z)=0$ implies $\E[ZV(Y/(\xi Z))]=0$ for some $Y\in\Y$. By positivity, this shows $ZV(Y/(\xi Z))=0$ a.s., from where $Z=0$ a.s. Finally, for axiom ($6$), first recall that $z \mapsto zV(Y/z)$ is increasing, from which $I(Z)\geq \sup_{0\leq \xi < 1} I(\xi Z)=:\zeta$. Now, take $\epsilon_n \nearrow 1$ so $\zeta=\lim I(\epsilon_n Z) $. Because $I$ is l.s.c.\ we deduce that $\lim I(\epsilon_n Z)\geq I(Z)$ and thus $I(Z)=\zeta$.\\
Now for $J$. Axioms ($1$), ($2$) and ($3$) are direct. If $J(\xi X)=0$ this means $YU^{-1}(\xi X)=0$, for all $Y\in\Y$ a.s.  Thus $X=0$ a.s. Lastly, by increasingness of $U^{-1}$ it holds that for fixed $Y$: $YU^{-1}(\xi X)\nearrow YU^{-1}(X)$ as $\xi\nearrow 1$. By monotone convergence then $\E[YU^{-1}(\xi X)]\nearrow \E[YU^{-1}(X)]$ and thus $\sup_{0\leq \xi<1}\E[YU^{-1}(\xi X)]=\E[YU^{-1}(X)]$ and now taking supremum over $Y\in \Y$ we get axiom ($6$).

\end{proof}


Call now $L_I^*$ and $L_J^*$ the topological duals. By the ``reflexivity Theorem'' in \cite{Nakano2} it holds automatically that the modulars $J$ and $I$ are reflexive, in the sense that if the following functionals are defined:
$$I^*(l):=\sup_{Z\in L_I}\{l(Z)-I(Z)\} \mbox{ for }l\in L_I^* \,\,\,\,\,\,\,\,\mbox{ and }\,\,\,\,\,\,\,\, J^*(j):=\sup_{X\in L_J}\{j(X)-J(X)\} \mbox{ for }j\in L_J^*,$$
then $I$ and $J$ may be recovered, that is:
$$I(Z)=\sup_{l\in L_I^*}\{l(Z)-L^*(l)\}\,\,\,\,\,\,\,\, \mbox{ and } \,\,\,\,\,\,\,\,J(X)=\sup_{j\in L_J^*}\{j(X)-J^*(j)\}.$$

 In particular then, both $I$ and $J$ are lower semicontinuous under the strong topologies introduced thus far, and by convexity, also under their weak topologies. What is more, from Lemma \ref{lema}, part 1), we deduce by Theorem 5.43 in \cite{Aliprantis} that both functionals are norm-continuous in the interior of their domains. \\
Another space of interest is the so-called set of finite elements of a modular space $L_F$, denoted $E_F$, which typically has better properties:
$$E_F=\{s\in {\cal S}:F(\alpha s)<\infty \mbox{ for all }\alpha>0\}.$$
We remark that $E_I=L_I=dom(I)$ as soon as condition \eqref{Vdelta} in Assumption \ref{Usuposdelta} holds.
Let us state now a few results that will be repeatedly useful:


\medskip 

\begin{lemma}
\label{lema}
For every $Z\in L_I$, $X\in L_J$:
\begin{enumerate}
\item $I\left(\frac{Z}{\vert Z\vert_I^l}\right)\leq 1$ and $J\left(\frac{X}{\vert X\vert_J^l}\right) \leq 1$. 
\item $Z_n$ norm converges to $Z$ in $L_I$ (respect.\ $X_n$ norm converges to $X$ in $L_J$) if and only if for all $\alpha>0$, $I(\alpha[Z_n-Z])\rightarrow 0$ (respect.\ $J(\alpha[X_n-X])\rightarrow 0$).
\item $I(Z)+J(X)\geq \E[XZ]$.
\end{enumerate}

\end{lemma}

\begin{proof}  We prove (1) first. Notice $J\left(\frac{X}{\vert X\vert_J^l}\right)\leq\sup_Y\E[YU^{-1}(X/\|X\|^l_{\eta_Y})]\leq 1$, the first inequality because clearly $\vert X\vert_J^l \geq \|X\|^l_{\eta_Y}$ and the second by definition of the Luxemburg norm and Fatou's Lemma. On the other hand take $\beta_n\searrow\vert Z\vert_I^l$ such that $I(Z/\beta_n)\leq 1$. Since $Z/\beta_n\rightarrow Z/\vert Z\vert_I^l$ in probability we conclude by Lemma \ref{Iconvex} that $$I\left(Z/\vert Z\vert_I^l\right)\leq \liminf I\left(Z/\beta_n\right)\leq 1.$$
Part (2) is a direct consequence of Theorem 3 in Chapter XI,81 of \cite{Nakano}.
For part (3), by Remark \ref{dualKoz} the conjugate of $\eta_Y$ is $\eta_Y^*$, and so $\E[XZ]\leq \E[ZV(Y/Z)]+\E[YU^{-1}(X)]$ for every $Y\in\Y^*$. Thus bounding $\E[YU^{-1}(X)]$ above by $J(X)$ and then taking infimum over $Y\in\Y$ yields $\E[XZ]\leq I(Z)+J(X)$.

\end{proof}


Time is ripe to prove some more refined properties of the spaces $L_I$ and $L_J$. Fortunately Lemma \ref{suposdensonec} says that the properties of both $L_Y$ and $L^*_Y$, with $Y\in\Y^*$, can be lifted.

\medskip 

\begin{proposition}
\label{lengthy}
Both subspaces $E_I$ and $E_J$ are closed subspaces of $L_I$ and $L_J$ respectively.
When considering the almost-sure ordering, $E_I$ and $L_J$ are Banach lattices, and furthermore $E_I$ is order-continuous.

\end{proposition}

In the last result, any of the previously defined norms may have been used. See the Appentix for the lengthy proof.

In order to further understand the modular spaces introduced thus far, and in doing so paving the way for the central statements of this section, some duality results will be pursued. First of all, H\"older-type inequalities are proved:

\medskip

\begin{proposition}
\label{Hoelder}
We have:
$$\vert\E[ XZ]\vert\leq \vert Z \vert_I^i \vert X \vert_J^j \leq 2 \vert Z \vert_I^k \vert X \vert_J^k, $$
where $i,j,k\in\{a,l\}$ and $i\neq j$.
Furthermore, the inclusions $L^{\infty}\rightarrow L_J\rightarrow L^1$ and $L^{\infty}\rightarrow L_I\rightarrow L^1$ are continuous.

\end{proposition}

\begin{proof} From inequality (3) in Lemma \ref{lema} follows that $\E[XZ]\leq \frac{1}{\alpha\beta}\{I(\alpha Z)+J(\beta X)\}$. Now, take $\beta$ such that $J(\beta X)\leq 1$. Then $\E[XZ]\leq \frac{1}{\beta}\left[\frac{1}{\alpha}\{1+I(\alpha Z)\}\right]$ and taking infimum over $\alpha >0$ yields $\E[XZ]\leq \frac{1}{\beta} \vert Z \vert_I^a$. Now taking infimum of the $1/\beta$ such that $J(\beta X)\leq 1$ gives $\E[XZ]\leq \vert X \vert_J^l \vert Z \vert_I^a$. From here also $\vert \E[XZ] \vert \leq \vert X \vert_J^l \vert Z \vert_I^a$ and by a similar argument $\vert \E[XZ] \vert \leq \vert X \vert_J^a \vert Z \vert_I^l$. Finally, because in the general context of modular spaces (see \cite{Nakano}, Chapter XI) holds that $ \vert\cdot\vert^l \leq \vert\cdot\vert^a \leq 2\vert\cdot\vert^l$ we get the desired inequalities. \\
Evidently $1\in L_J$ and by Assumption \ref{assgnral} also $1\in L_I$. By using the derived H\"older inequalities, this shows the continuity of the inclusions into $L^1$. On the other hand, because both $I$ and $J$ are increasing, $\vert \cdot \vert_I\leq \vert\cdot \vert_{\infty}\vert 1\vert_I$ and likewise for $J$, thus proving the continuity of the inclusions from $L^{\infty}$.
%

\end{proof}


Notice from this that, as it can be expected, for every $X\in L_J$ the functional $l_X(\cdot)=\E[\cdot X]$ belongs to $L_I^*$ and for every $Z\in L_I$ the functional $l_Z(\cdot)=\E[\cdot Z]$ belongs to $L_J^*$. 
We state now a Riesz-type representation result. This will rest in a few technical points to be established in Lemma \ref{lemR}. Both proofs are given in the Appendix.


\begin{proposition}
\label{rep}
The topological dual of $E_I$ is $L_J$, with the usual identification:

$$l \in (E_I)^* \leftrightarrow l(Z)=\E[ZX] \mbox{ for some }X\in L_J,$$

and this identification is isomorphic isometric between $(E_I,\vert\cdot\vert_I^a)$ and $(L_J,\vert\cdot\vert_J^l)$.\\
Furthermore, for every $Z\in L_I,X\in L_J$, we have $I^*(l_X)=J(X)$, and if $E_I=L_I$ also $J^*(l_Z)=I(Z)$.

\end{proposition}

\medskip

\begin{lemma}
\label{lemR}
\quad
\begin{enumerate}

\item $\ind{A}\in E_I$ for every $A\in\F$

\item Simple functions are norm dense in $E_I$

\item If $Z_n\rightarrow 0$ a.s. and $\vert Z_n\vert$ is bounded by a constant, then $\vert Z_n\vert_I\rightarrow 0$

\item If $\kappa:=\sup\{\vert\E[fg] \vert :f\mbox{ simple and }\vert f \vert_I^a\leq 1 \}<\infty$ then $g\in L_J$ and $\vert g \vert_J^l = \kappa$

\end{enumerate}

\end{lemma}

\medskip

Notice that   a property  analogous to  point $(3)$ in the above lemma does not hold in $E_J$ if $\Y$ is not uniformly integrable. 


\subsection{Applications of the modular approach to the robust optimization problem}\label{subsecmain}

As a consequence of   Proposition \ref{Hoelder}, we can prove the following result, of interest on its own, which we already mentioned  in Remark \ref{ilusion} and will be useful in proving the general minimax Theorem \ref{minimaxsinreflexividad} below:
\begin{proposition}
\label{coercividad2}
Under Assumption  \ref{assgnral},  for all  $x>0$ we have that  
\begin{equation} 
\forall  \Q\in \qu:\quad (1+x)\left | \frac{d\Q}{d\R}\right |^l_{I} \geq  u_{\Q}(x)\geq (1\wedge x) \left | \frac{d\Q}{d\R}\right |^a_{I} . \label{sandw2}
\end{equation}

\end{proposition}

\begin{proof}
By Proposition \ref{Hoelder} we have:
$$\E^{\Q}[U(X_T)]\leq |d\Q/d\R |_I^l |U(X_T) |_J^a \leq [1+J(U(X_T))] |d\Q/d\R |_I^l,$$
by definition of the norm. Hence, by Lemma \ref{Iconvex} we get that $u_{\Q}(x)\leq [1+x] |d\Q/d\R |_I^l$.\\ 
Now we prove the lower bound for $u_{\Q}(x)$ in \eqref{sandw2}. Let us call $Z=\frac{d\Q}{d\R}\in\frac{d\qu_e}{d\R}$. Recalling that $v_{\Q}(y):=\inf_{Y\in\Y}\E\left[ZV(yY/Z)\right]$, we have: 
$$|Z|^a_{I} \leq y +  yI(Z/y)  = y + v_{\Q}(y) \leq y + c v_{\Q}(y),$$ for each $c\geq 1$. Calling $A_{\Q}(y)=v_{\Q}(y) + xy$, then $A_{\Q}(y) \geq \frac{1}{c}|Z|_{I} + \left(x-\frac{1}{c}\right)y $. Thus for every $x>0$, finding $c \geq 1$ such that $x\geq c^{-1}$ and then taking infimum over $\{y>0\}$ yields $u_{\Q}(x)\geq C|Z|_{I}$: if the r.h.s.\ is infinite there is nothing to prove, and otherwise by {\it Theorem 3.1} in \cite{KrSch} it holds $u_{\Q}(x)=\inf_{y>0}\left[ v_{\Q}(y) +xy \right]$ and we still get the desired bound. The best constant $C$ is thus $1\wedge x$.\\
If now $Z:=d\Q/d\R \in \frac{d\qu}{d\R}\setminus \frac{d\qu_e}{d\R} $, an easy application of 
\textit{Lemma 3.3} in \cite{SchiedWu} allows to conclude, from the previous bounds.
\qquad 
\end{proof}

%

\medskip

Thanks to Proposition \ref{rep} we can endow $L_J$ with a decent weak-* topology and thus finally prove one of our main results for incomplete markets: Theorem \ref{minimaxsinreflexividad}.\\

\begin{proof}{\textit{of Theorem \ref{minimaxsinreflexividad}}}
Fix $x>0$. We intend to apply Theorem 7, chapter 6, in \cite{AubEke} (Lopsided minimax Theorem, also stated on page 295 therein). First, let us define the set $G:=\{g\in L_J:0\leq g \leq U(X_T)\mbox{, some }X\in \mathcal{X}(x)\}$. Now we define a bilinear function $F:G\times d\qu/d\R\rightarrow [0,\infty)$   by $F(g,Z)=\E[Zg]$. Evidently under condition $L_I^*\cong L_J$ we must have that $E_I=L_I$ (which is the case anyway if condition \eqref{Vdelta} in Assumption \ref{Usuposdelta} holds).\\
We first endow the convex set $G$ with the weak-* topology $\sigma(L_J,E_I)$. Let us prove that $G$ is closed with it. Indeed if $\{g_{\alpha}\}_{\alpha}\subset G$, we have by Lemma \ref{Iconvex}, part c), that $J(g_{\alpha})\leq x$. But by Proposition \ref{rep}, the spaces $(E_I,\sigma(E_I,L_J)), (L_J,\sigma(L_J,E_I))$ are in topological duality and $J=I^*$. Therefore $J$ is $\sigma(L_J,E_I)$-l.s.c.\ and we conclude that if $g_{\alpha}\rightarrow g$ in this topology, then $J(g)\leq x$. Again by Lemma \ref{Iconvex}, part c), we see that $|g|\in G$. On the other hand $\ind{g<0}\in E_I$ (by Lemma \ref{lemR}) and so $\E[g\ind{g<0}]=\lim \E[g_{\alpha}\ind{g<0}]\geq 0$, from which $g\geq 0$ and so $g\in G$.\\
We now prove that $G$ is weak*-compact. By Banach-Alaoglu it suffices to prove that it is norm bounded. But this holds since $|g|^a_J\leq 1+J(g)\leq 1+x$, for every $g\in G$.\\
We apply now the lopsided minimax Theorem. The function $F$ satisfies:
\begin{itemize}
\item $F(g,\cdot)$ is convex
\item $\{g\in G:F(g,Z)\geq \beta\}$ is weak*-compact for every $\beta,Z$.
\item $F(\cdot,Z)$ is concave and continuous,
\end{itemize}
and thus $-F$ satisfies with ease the requirements of that theorem. We conclude then the minimax equality and the attainability of an optimal $g\in G$. By simple arguments in \cite{SchiedWu} (see the proof of Lemma 3.4 therein) any optimal $g$ must be of the form $U(X_T)$ and one may approximate the \textit{infsup} by taking the infimum over $\qu_e$.\\
Because we proved that $u(x)=\inf_{\Q\in\qu_e} u_{\Q}(x)$ we also have $u(x)=\inf_{\Q\in\qu_e,u_{\Q}(x)<\infty} u_{\Q}(x)$. Now applying Theorem 3.1 in \cite{KrSch} we see that $u(x)=\inf_{y\geq 0 }\left[\inf_{\Q\in\qu_e,u_{\Q}(x)<\infty} v_{\Q}(y)+xy\right]$ and so by the first statement in Lemma 3.5 in \cite{SchiedWu} we conclude that $u$ is the conjugate of $v$. Finiteness of $v$ on $(0,\infty)$ is a consequence of $L_I=E_I$. Because $I$ is convex and $v(y)=\inf_{Z\in d\qu_e/d\R}yI(Z/y)$, an argument as in the proof of Lemma \ref{lemagama} shows that $v$ is convex and so we conclude by Theorem 7.22 in \cite{Aliprantis} that $v$ is continuous in $(0,\infty)$. Since clearly $v(y)\geq V(y)$ we see that $v(0+)=\infty$. Thus defining $v(\cdot)=\infty$ on $(-\infty,0]$ we get a l.s.c.\ function everywhere. Defining $u(0)=0$ and $u(x)=-\infty$ if $x<0$, we still get that $u$ is the concave conjugate of $v$. This in turn implies that $v$ is conjugate to $u$ and also that if $y>0$ then $v(y)=\sup_{x>0}[u(x)+xy]$.\\
%
%
Finally, in the reflexive case, when computing $\inf_{\Q \in \qu}  \E^{\Q}\left(U\left(\hat{X}_T\right)\right)$ we realize that it is enough to do it over a norm-bounded subset of $d\qu/d\R$. Indeed, we have already proven that $u(x)=\inf_{\Q\in\qu}u_{\Q}(x)$, and this is finite by Assumption \ref{assgnral}. Thus we may only regard $\qu\cap \{\Q:u_{\Q}(x)\leq u(x)+1\}$, but by Proposition \ref{coercividad2} we have that $u_{\Q}(x)\geq c(x)|d\Q/d\R|_I^a$, and so this set is contained in $\qu\cap \{\Q:|d\Q/d\R|_I^a\leq c(x)^{-1}[u(x)+1]\}$. By reflexivity and Assumption \ref{assgnral}, these sets are weakly compact (i.e.\ $\sigma(E_I,L_J)$-compact) and so the continuous linear functional $Z\mapsto \E\left(Z U\left(\hat{X}_T\right)\right)$ attains its minimum there. Any of these densities along with the optimal $\hat{X}$ conforms a saddle point. We finally stress that the reflexivity condition on $L_I$ is satisfied if the market is complete and Assumption \ref{Usuposdelta} holds. Indeed by completeness we would have that $I(\cdot)=\E[\eta^*_1(\cdot)]$ and $J(\cdot)=\E[\eta_1(\cdot)]$, and so by Assumption \ref{Usuposdelta} coupled with Proposition \ref{quienes} and Theorem \ref{OMreflex} we get the desired reflexivity.
\\
 \qquad 
\end{proof}

\smallskip

\begin{remark} From the previous proof it is clear that if $d\qu/d\Prob\subset E_I$ then at least for the minimax result and the existence of an optimal wealth, the condition $L_I^*\cong L_J$ can be avoided altogether, since we may work with $E_I$ instead of $L_I$ from the beginning, and $E_I^*\cong L_J$ holds. 
\end{remark}

\smallskip

Let us point out that at the moment we can only prove existence of a worst-case $\hat{\Q}$ (as well as relating it explicitly to the optimal $\hat{X}$) in the case that our modular spaces are reflexive. In Theorem \ref{triste} and Remark \ref{obslimites}, we aim to find out when this is the case.
The following property relates the answer to  the set  $\Y$.

\medskip
\begin{lemma}
\label{propUI}
If $E_J$ has order-continuous norm (i.e. $\vert x_{\alpha}\vert_J \searrow 0$ whenever $x_{\alpha}\searrow 0$) then $\Y$ is uniformly integrable.
\end{lemma}

\begin{proof}
By Theorem 9.22 in \cite{Aliprantis}, $E_J$ has order-continuous norm if and only if every sequence of order-bounded and disjoint
elements is strongly convergent to zero. So take $A_n$ a sequence of disjoint sets. Notice that $\ind{A_n}$ is an order-bounded and disjoint sequence, and thus $\vert \ind{A_n}\vert_J\rightarrow 0$. This implies $J(\ind{A_n})\rightarrow 0$, which means $\sup_{Y\in\Y}\E[\ind{A_n}Y]\rightarrow 0$. Now, from Theorem 7 in \cite{Diestel} this implies that $\Y$ is uniformly integrable. 

\end{proof}

\medskip

\noindent  The following theorem is essential and it implies Theorem \ref{tristeintro}.  

\begin{theorem}
\label{triste} If the set $\Y$ is not uniformly integrable, then neither $E_J$, $L_J$ nor $E_I$ can be reflexive.

\end{theorem}

\begin{proof} As pointed out in Corollary 9.23 in \cite{Aliprantis}, a reflexive Banach lattice has order continuous norm. Since $E_J$ is a Banach lattice in itself, if it were reflexive, by Lemma \ref{propUI} the set $\Y$ would be uniformly integrable. Thus $E_J$ is not reflexive and therefore $L_J$ neither, since the former is a closed subset of the latter. On the other hand, under the assumption of this section the dual of $E_I$ is isomorphic to $L_J$ (which we proved in Proposition \ref{rep}) which in turn implies that $E_I$ cannot be reflexive either. 

\end{proof}

\medskip

\begin{remark}The previous result  states that lack of uniform integrability of $\Y$ implies that the space $L_I$ cannot be reflexive. This means that the approach used for Orlicz-Musielak spaces (in the complete case) does not extend vis-\`{a}-vis to the current modular space setting. It is remarkable that no growth conditions on $U$ or $V$ may yield reflexivity to our modular spaces as soon as $\mathcal{Y}$ is not uniformly integrable. 
\end{remark}

 \medskip

\begin{remark}
\label{obslimites}
If the set $\Y$ were uniformly integrable, then also the set of absolutely continuous martingale measure $\mathcal{M}$ would be so (more precisely, their densities would be $\sigma(L^1,L^{\infty})-$ relatively compact). Theorem 6.7 and Corollary 7.2 in \cite{Delb} then say that $\mathcal{M}$ must be a singleton, at least in the case of bounded continuous prices and either if all martingales on the filtration are continuous (e.g.\ the augmented brownian filtration) or if the filtration is quasi left-continuous. Therefore in most cases {uniform integrability of }$\Y$ {implies completeness}.

\end{remark}

\medskip

We envisage that further analysis of our modular spaces (for instance identifying the dual of $L_J$, or establishing when $L_I$ is a norm-dual space) may bring a better understanding of the robust problem and the (non)existence of the associated worst-case measures. This could be endeavoured through minimization of entropy techniques alternatively. 

\section*{Appendix}

\begin{proof}{\it (Lemma \ref{lemagama})}
The first two items are well-known and can be found in Lemma 2.3.2 in \cite{JBtesis}. We prove only the third one here. Clearly $\bar{\gamma}_l(x)=\sup_{z\geq 0}\{\vert x\vert z-z V(l/z)\}$. The first order condition for this (assuming $z\neq 0$) is $\vert x\vert-V(l/z)+\frac{l}{z}V'(l/z)=0$. But using that $V'=-[U']^{-1}$ one gets $\vert x\vert=U([U']^{-1}(l/z))$ or better $z=\frac{l}{U'\circ U^{-1}(\vert x\vert)}$. Therefore $\bar{\gamma}_l(x)=\frac{\vert x\vert l}{U'\circ U^{-1}(\vert x\vert)}-\frac{l}{U'\circ U^{-1}(\vert x\vert)}V\circ U'\circ U^{-1}(\vert x\vert)$. Using again the identity $V(y)=U([U']^{-1}(y))-y[U']^{-1}(y)$ one arrives at $\bar{\gamma}_l(x)=lU^{-1}(\vert x\vert)$. By Lemma \ref{lemagama} one knows that $\bar{\gamma}_l\geq 0$ and is null only at the origin. Thus if the supremum defining it were attained at $0$, since $0V(l/0)=0$, this shows $x$ must be null. But also $U^{-1}(0)=0$. Hence, the asserted expression for $\bar{\gamma}_l$ is always valid.
\qquad 
\end{proof}

\smallskip

\begin{proof}{\it (Lemma \ref{domPhi})}  Let $\widetilde{L_{\eta^*}}$ denote the  algebraic dual  of   $L_{\eta}$ and $\widehat{L_{\eta^*}}$  its subspace of relatively bounded forms.  We extend $\Phi_y^*$  to $\widetilde{L_{\eta^*}}$ by replacing   the expectation in \eqref{Phi*yZ}     by the  dual product in  $\widetilde{L_{\eta^*}}\times  L_{\eta}$  and note that this $\Phi_y^*$  
  corresponds to the function $\Phi^* $ in  Proposition 5.10 in  \cite{Leo-ent}, while space $U$ therein corresponds to  space $L_{\eta}$ here. Moreover,  $\Phi^*_+$  and $\Phi^*_-$   therein respectively  correspond in our setting to $\Phi^*_{y,+}$ 
and the convex indicator of $0$ (since $\gamma(-|\cdot|)=0$) and part a) of that result we then get   $\mbox{dom} \, \Phi^*_{y}=\{ \xi \in \widetilde{L_{\eta^*}} \, : \Phi^*_{y}(\xi)<\infty\} \subseteq \{ \xi \in \widehat{L_{\eta^*}}:\, \xi_-=0\} $ 
and   
$\Phi^*_{y}(\xi)=\Phi^*_{y,+}(\xi_+)=\Phi^*_{y}(\xi_+) $ for all $\xi \in \mbox{dom} \, \Phi^*_{y}$. 

Notice  now  on one hand  that  $L_{\eta^*}\subset \mbox{dom} \, \Phi^*_{y,+}$  since $ \Phi^*_{y,+}(Z)= \int \gamma_y^*(|Z| ) d\R<\infty$   for $Z\in L_{\eta^*}$ and, on the other, 
$ \langle \xi, W/ \| W\|_{L_{\eta}} \rangle  \leq  \Phi^*_{y,+}(\xi) +    y \int \gamma(|W |/\| W\|_{L_{\eta}} ) d\R\leq  \Phi^*_{y,+}(\xi) +    y $
 for all $\xi\in  \mbox{dom} \, \Phi^*_{y,+}$ and $W\in  L_{\eta}\backslash\{0\}$, 
since  $\int \gamma(|W |/\| W\|_{L_{\eta}} ) d\R=\int \eta(W /\| W\|_{L_{\eta}} ) d\R\leq 1$ (by definition of $ \| W\|_{L_{\eta}} $ and Fatou's Lemma). Taking $-W$ instead of $W$, we  get 
 $| \langle \xi, W \rangle | \leq  (  \Phi^*_{y,+}(\xi) +    y )\| W\|_{L_{\eta}}$. Thus,   we have 
$\mbox{dom} \, \Phi^*_{y,+}=L_{\eta^*}$ and, in the notation of   Proposition 5.10 in  \cite{Leo-ent},   $L=L_{\eta^*}$, $L_+=L_{\eta^*}$ and $L_-= \{0\}$.
With  part b) of  that result we get that
$\mbox{dom} \, \overline{\Phi}_{y}\subset \widehat{L_{\eta}}$
and that  for all $\zeta \in \mbox{dom} \,  \overline{\Phi}_{y}$ the first two equalities in \eqref{phiphiphi} hold.  Since the Orlicz space $L_{\eta^*}$ is reflexive, by Theorem 9.11 in \cite{Aliprantis} we get that  $\widehat{L_{\eta}}=L_{\eta} $ so that  $\mbox{dom} \, \overline{\Phi}_{y}\subset  L_{\eta}$  as claimed.  We then easily conclude  since $ \overline{\Phi}_{y}$ coincides with $\Phi_{y}$ on $L_{\eta}$. 
 \end{proof}
\smallskip

\begin{proof}{\it (Lemma \ref{suposdensonec})}
We prove (i) first. Call $Y^*$ some element of $\Y^*$. For any $Y \in \Y$ define $Y^n=\frac{n-1}{n}Y+\frac{1}{n}Y^*$. By convexity $Y^n \in \Y$, and by non-negativity $Y^n\geq\frac{1}{n} Y^*$, implying that $Y^n\in \Y^*$, since $V$ is decreasing. 
By convexity $\E[\vert Z \vert V(Y^n_T /\vert Z \vert ) ]\leq \left(\frac{n-1}{n}\right )\E[\vert Z \vert V(Y_T /\vert Z \vert ) ] + \frac{1}{n}\E[\vert Z \vert V(Y_T^*/\vert Z \vert ) ] $, so $\liminf \E[\vert Z \vert V(Y^n_T /\vert Z \vert ) ]\leq \E[\vert Z \vert V(Y_T /\vert Z \vert ) ]$, and we get that $I(Z)=\inf_{Y\in \Y^*}\E[\eta_Y^*(Z)]$. On the other hand, take $X\in dom(J)$ and since of course $\frac{n-1}{n}\E[YU^{-1}(X)]+\frac{1}{n}\E[Y^*U^{-1}(X)]$ tends to $\E[YU^{-1}(X)]$, this directly shows that $J(X)=\sup_{Y\in \Y^*}\E[\eta_Y(X)]$.
If $J(X)=+\infty$, take $\E[\hat{Y}_mU^{-1}(X)]$ growing to $+\infty$. If these values are finite then the previous argument shows how to approximate them in $\Y^*$. If (for large enough $m$) they are infinite, then also $\frac{n-1}{n}\hat{Y}_m + \frac{1}{n}Y^*$ generates an infinite value. Therefore the identity for $J$ always holds.\\
For condition (ii), one need only observe that $1\in\Y^* $ and $\mathcal{E}\left(-\int\lambda dM\right )\in\Y^* $, respectively.
\end{proof}

\medskip

\begin{proof}{\it (Proposition \ref{lengthy})}
The almost-sure order is a partial order. From this both $L_I$ and $L_J$ are ordered vector spaces and lattices, that is, Riesz lattices. Now, because any of the norms defined in this section are lattice norm (i.e.\ order preserving), both $L_I$ and $L_J$ are Normed Riesz Spaces. \\
First we prove that both $E_I$ and $E_J$ are closed subspaces of $L_I$ and $L_J$, in the spirit of the proof of Proposition 3 in \cite{RaoRen}, Chap.  3.4. Denote $F$ either $I$ or $J$. We need to show that $\overline{E_F}\subset E_F$. Take $s \in \overline{E_F}$ and $s_n\rightarrow s$ elements in $E_F$. For a fixed positive $k$, choose $n$ so that $\vert s-s_n \vert_F^l < \frac{1}{2k}$. We then see by convexity and Lemma \ref{lema} part 1), that 
$$F(2k[s-s_n] )=F\left ( \frac{2k[s-s_n] \vert 2k[s-s_n] \vert_F^l}{\vert 2k[s-s_n] \vert_F^l} \right )\leq \vert 2k[s-s_n] \vert_F^l \leq 1.$$  

Thus, since $ks=\frac{1}{2}(2k[s-s_n])+ \frac{1}{2}[2ks_n] $ we get by convexity that $F(ks)\leq \frac{1}{2} F(2k[s-s_n])+\frac{1}{2}F(2ks_n)<\infty$. Since this holds for any $k>0$, we conclude that $s\in E_F$. 

Now completeness of $E_I$ and $L_J$ will be proved, showing that both spaces are Banach lattices.
For $E_I$ recall (Theorem 9.3 in \cite{Aliprantis}) that a Normed Riesz space is a Banach Lattice if and only if every positive, increasing Cauchy sequence is norm convergent. Therefore take $(Z_n)$ a positive, increasing Cauchy sequence in $E_I$ (for Luxemburg's norm). By definition $(Z_n)$ converges a.s.\ to its supremum, which we call $Z$, and might be $\infty$-valued. Since the sequence is Cauchy, there is a $k>0$ such that $\vert Z_n\vert_I^l\leq k$ for every $n$. By parts (1) and (3) in Lemma \ref{lema} we have that $\E(Z_n/k)\leq I(Z_n/k)+J(1)\leq 1+U^{-1}(1)$ implying by Fatou's Lemma that $Z$ is in particular finite, and so $Z_n$ converges to $Z$ in probability (on the non-extended real line). Notice that for every $\lambda>0$ also $I(\lambda (Z_n-Z_m))\rightarrow 0$ as $(n,m)$ grows. Indeed, if $\lambda\vert Z_n-Z_m \vert_I^l\leq \epsilon <1$ we have by convexity and Lemma \ref{lema}.(1) that $I(\lambda (Z_n-Z_m))\leq \lambda \vert Z_n-Z_m \vert_I^l \leq  \epsilon $. Thus,  fixing any $\lambda>0$ we have for every $\epsilon>0$ the existence of $N=N(\lambda,\epsilon)$ big enough s.t. $m>n>N$ implies $I( \lambda (Z_m-Z_n))\leq \epsilon$ and hence taking limit in $m$ by lower-semicontinuity we get $I(\lambda (Z-Z_n))\leq \epsilon$. Therefore $I(\lambda|Z_n-Z|)\rightarrow 0$ and by part (3) in Lemma \ref{lema} we see that $Z_n\rightarrow Z$ strongly. By the first part of this proof we finally get that $Z\in E_I$.

Now for $L_J$, take $(X_n)$ an arbitrary Cauchy sequence. The same sequence is Cauchy in every Orlicz-Musielak space associated to $YU^{-1}(\cdot)$ ($Y\in\Y^*$). Call $\|\cdot\|_Y$ the associated Luxemburg norm. Because these spaces are complete, the sequence norm-converges to (possibly different) limits in each of them. However, since this convergences are stronger than $L^0$ convergence, the limit must be necessarily (a.s.) unique. Thus, $X_n\rightarrow X$ for every Orlicz-Musielak space associated to $\eta_Y$ and in probability. By Fatou's lemma $W\mapsto \E[YU^{-1}(W)]$ is lower-semicontinuous in $(L^0)_+$ and thus (as a supremum) also $J(\cdot)$ is so, from which $J(kX)\leq \liminf J(kX_n)\leq 1$ where $k^{-1}$ is an upper bound for the $L_J$ norms of the $(X_n)$ (it exists because sequence is Cauchy) and by Lemma \ref{lema}.(1). Therefore $X\in L_J$. Evidently $\| X_n-X\|_Y \leq \| X_n-X_m\|_Y + \| X_m-X\|_Y \leq \vert X_n-X_m\vert_J^l + \| X_m-X\|_Y $.  Now given $\epsilon >0$ we can make $\vert X_n-X_m\vert_J^l\leq \epsilon$ for $n,m\geq N$ independently of $Y\in\Y^*$. On the other hand $\| X_m-X\|_Y\leq \epsilon$ for $m\geq M(Y)$. From here, $\| X_n-X\|_Y\leq 2\epsilon$ for every $n\geq N$ independent of $Y$. Thus by Lemma \ref{lema}.(1) again, $\E[YU^{-1}([X_n-X/[2\epsilon]])]\leq 1$ and taking supremum yields $J([X_n-X]/[2\epsilon])\leq 1$ also, from which $\vert X_n-X\vert_J^l\leq 2\epsilon$ by definition of this norm. Therefore the sequence is convergent. 

For the order-continuity of $E_I$, we need to show that if $Z_{\alpha}\searrow 0$ a.s. then $\vert Z_{\alpha} \vert_I \searrow 0$. Fix $\beta >0$ and for a fixed $\alpha_0$ in the set of indices, notice that $I(\beta Z_{\alpha_0})<\infty$. Thus there is a $Y$ such that $\E[Z_{\alpha_0} V(Y/(\beta Z_{\alpha_0}))]<\infty$. But $Z_{\alpha} V(Y/(\beta Z_{\alpha}))$ decreases to $0$ and is dominated by $Z_{\alpha_0} V(Y/(\beta Z_{\alpha_0}))$ (for $\alpha$ big enough, in the sense of the net), which is integrable. By dominated (or monotone) convergence then $\E[Z_{\alpha} V(Y/(\beta Z_{\alpha}))]\searrow 0$ and therefore $I(\beta Z_{\alpha})\searrow 0$. Since this holds for every $\beta >0$, by Lemma \ref{lema}.(3) this shows that $\vert Z_{\alpha} \vert_I \searrow 0$.
%

\end{proof}

\medskip

\begin{proof}{\it (Proposition \ref{rep})}
Let $l\in (E_I)^*$ and define $\mu(A):=l(\ind{A})$ for $A\in\F$ (well-defined and finite by Lemma \ref{lemR},(1)). Clearly $\mu(\emptyset)=0$. Also if $A_n\in\F$ are disjoint, and writing $A=\cup_n A_n$, then $\sum_{n\leq N} \ind{A_n}\rightarrow \ind{A}$ a.s. and $\vert \sum_{n\leq N} \ind{A_n} - \ind{A} \vert \leq 1$. Therefore by (3) in Lemma \ref{lemR} then $\sum_{n\leq N} \ind{A_n}\rightarrow \ind{A}$ in $E_I$. By continuity of $l$ then $l(\ind{A})=\lim_{N}\sum_{n\leq N} l(\ind{A_n})$. Thus $\mu$ is clearly a finite, signed, countably-additive measure. If $A\in\F$ is such that $\Prob(A)=0$ then $l(\ind{A})=0$ and hence $\mu(A)=0$. So $\mu$ is absolutely continuous w.r.t.\ $\Prob$. By Radon-Nikodym's Theorem then $g:=\frac{d\mu}{d\Prob}$ exists and is $\Prob-$integrable. By linearity then $l(f)=\E[fg]$ for every simple function $f$. By continuity $\vert\E[fg] \vert \leq C \vert f \vert_I$ for simple functions. Therefore $\sup\{\vert\E[fg] \vert :f\mbox{ simple and }\vert f \vert_I^a\leq 1 \}<\infty$ and by (4) in Lemma 
\ref{lemR} we get that $g\in L_J$ and that $\vert g \vert_J^l$ equals the above supremum. Since both $l(\cdot)$ and $\E(\cdot g)$ are uniformly continuous functions coinciding on a dense set (by (2) in Lemma \ref{lemR}, simple functions are such a set), they must agree in the whole of $E_I$.
Hence $l(f)=\E[fg]$ for every $f\in E_I$ and so $(E_I)^*\subset L_J$, but using Proposition \ref{Hoelder} the reverse inclusion already holds. Therefore $(E_I)^* = L_J$, where the identification is isomorphic if $L_J$ is endowed with the Luxemburg norm and $E_I$ with the Amemiya one.\\
Now take $X\in L_J$ and call $l_X(\cdot):=\E[X\cdot]$. Then: \small
$$
\begin{array}{lllll}
I^*(l_X)&=& \sup\limits_{Z\in L_I}\left\{ \E[XZ] - \inf\limits_{Y\in\Y^*}\E\left[\vert Z\vert V\left(\frac{Y}{\vert Z\vert}\right )\right]\right \}
&=& \sup\limits_{\Y\in\Y^*}\sup\limits_{Z\in L_I}\left\{ \E[XZ] - \E\left[\vert Z\vert V\left(\frac{Y}{\vert Z\vert}\right )\right] \right\} \\
&=& \sup\limits_{\Y\in\Y^*}\sup\limits_{Z\in L_{\eta^*_Y}}\left\{ \E[XZ] - \E\left[\vert Z\vert V\left(\frac{Y}{\vert Z\vert}\right )\right]\right \} 
&=& \sup\limits_{\Y\in\Y^*} \E[YU^{-1}(X)]\\
&=& J(X)&&,
\end{array}
$$\normalsize
since the conjugate of $\eta_Y^*$ is $\eta_Y$. Now fix $Z\in L_I$ and assume $L_I=E_I$. Then $J^*(l_Z)=\sup_{X\in L_J}\{\E[XZ]-I^*[l_X]\}$ by the previous lines. On the other hand,
$I(Z)=\sup_{l\in (L_I)^*}\{l(Z)-I^*(l)\}=\sup_{X\in L_J}\{\E[XZ]-I^*[l_X]\}$, since $(E_I)^*=L_J$. Thus $J^*(l_Z)=I(Z)$.

\end{proof}

\medskip

\begin{proof}{\it (Lemma \ref{lemR})}
For the first point, $\ind{A}\in E_I$ iff $\inf_{Y\in\Y}\E[\ind{A}V(\beta Y)]<\infty$ for every $\beta >0$. This is true, simply by taking a $Y\in\Y^*$.
 

For the third point, if $\vert Z_n\vert \leq K$, then $I(\alpha Z_n)\leq \alpha \inf_{Y\in\Y}\E[KV(Y/(\alpha K))]$. But we have $\vert Z_n \vert V(Y/\alpha\vert Z_n \vert)\rightarrow 0$ a.s.\ and this sequence is dominated by $KV(Y/(\alpha K))$. Therefore if there exists a $Y\in\Y$ such that $\E[V(Y/(\alpha K))]<\infty$, then it would follow that $I(\alpha Z_n)\rightarrow 0$. But this holds (for every $\alpha >0$) again by taking $Y\in\Y^*$. By Lemma \ref{lema}.(3) we conclude that $Z_n\rightarrow 0$ strongly.

The proof of the second point resembles the previous one. First, since simple functions are dense in $L^{\infty}$ and by Proposition \ref{Hoelder} this last space is contained continuously in $L_I$ (obviously then also in $E_I$), it suffices to show that bounded functions are dense in $E_I$. Take $Z\in E_I$ and define $Z_n=Z\ind{\vert Z \vert <n}$. Thus $X_n:=\vert Z-Z_n \vert=\vert Z\vert \ind{\vert Z \vert \geq n}\searrow 0$ a.s. Now fix $\beta >0$. Taking any $N>0$ and because $\infty > I(\beta X_N)=\beta\E[X_N V(Y/(\beta X_N))]$ for some $Y\in\Y$, and $X_n V(Y/(\beta X_n))\searrow 0$ a.s. then by dominated (or monotone) convergence $\E[X_n V(Y/(\beta X_n))]\rightarrow 0 $ and thus $I(\beta X_n) \rightarrow 0 $. Now because this holds for every $\beta$, by Lemma \ref{lema}.(3) then $\vert X_n \vert_I\rightarrow 0$.

Finally, for the fourth point, take $\kappa<\infty$ as in the statement. Then clearly $\sup\{\vert\E[zg] \vert :z\mbox{ simple and }\vert\vert z \vert\vert_{\eta^*_Y}^a\leq 1 \}\leq \kappa$ for every $Y\in\Y^*$. A classical result in Orlicz theory (see (10) in Proposition 10, \cite{RaoRen}, chapter 3.4), which readily generalizes to Orlicz-Musielak spaces, implies that $\| g \|_{\eta_Y}^l = \sup\{\vert\E[zg] \vert : \vert\vert z \vert\vert_{\eta^*_Y}^a\leq 1 \}$, and hence $\| g \|_{\eta_Y}^l\leq \kappa$, since any non-negative $z$ may be approximated in an increasing way a.s. by simple functions. Hence $\sup_{Y\in\Y^*} \| g \|_{\eta_Y}^l \leq \kappa$. Since $\E[YU^{-1}(g/\| g \|_{\eta_Y}^l)]\leq 1$ (by definition of the norm and Fatou's Lemma) then $\E[YU^{-1}(g/\kappa)]\leq 1$ and thus $J(g/\kappa)\leq 1$ from which $\vert g \vert_J^l\leq \kappa<\infty$. Finally, by Proposition \ref{Hoelder} we have $\vert\E[fg] \vert \leq \vert g \vert_J^l \vert f \vert_I^a$ and so if $f$ is simple and such that $\vert f \vert_I^a\leq 1$ we get that $\vert\E[fg] \vert \leq \vert g \vert_J^l $ and then by taking supremum over such functions we derive that $\kappa \leq \vert g \vert_J^l $, and therefore there is equality.
%

\end{proof}

\smallskip

{\bf Acknowledgements: } The authors are grateful to  two anonymous referees for their careful  readings and for  valuable comments and suggestions that lead to substantial improvements in  the presentation. They also
 thank C.L\'eonard for useful exchanges about his works on 
entropy minimization and for providing full details on a gauge argument 
therein, F.Delbaen for 
kindly commenting on the link between compactness of martingale measure densities and incompleteness,  
A.Schied for encouraging comments  on  an early version of this work  and for pointing out some references, M.Kupper for useful insight into the incomplete setting and 
A. Kozek for  providing electronic versions of his works. JB  (resp. JF) acknowledges hospitality 
and support  of HIM Institute in Bonn, where part of this work was 
carried out during a one month (resp. one week) visit in May 2013 in the context of the trimester program Stochastic Dynamics in Economics and Finance. 

\bibliographystyle{plain}
\bibliography{mybib}

\end{document}